\documentclass[10.5pt,a4paper]{article}

\usepackage{graphicx,latexsym,euscript,makeidx,color,bm}
\usepackage{amsmath,amsfonts,amssymb,amsthm,thmtools,mathtools,mathrsfs,enumerate}
\usepackage[colorlinks,linkcolor=blue,citecolor=red]{hyperref}
\usepackage{geometry}
\allowdisplaybreaks[4]
\def\dbE{\mathbb{E}}    
\def\dbF{\mathbb{F}}  \def\cF{{\cal F}}  
    
\def\dbH{\mathbb{H}}  \def\cH{{\cal H}}  
    
  \def\cJ{{\cal J}}  
  \def\cK{{\cal K}}  
  \def\cL{{\cal L}}  
  \def\cM{{\cal M}}  
  \def\cN{{\cal N}}  
  \def\cO{{\cal O}}  
\def\dbP{\mathbb{P}}    
  \def\cQ{{\cal Q}}  
\def\dbR{\mathbb{R}}  \def\cR{{\cal R}}  
\def\dbS{\mathbb{S}}  \def\cS{{\cal S}}  
    
  \def\cU{{\cal U}}

  \def\cX{{\cal X}}

\def\ss{\smallskip}     \def\lt{\left}        \def\hb{\hbox}
\def\ms{\medskip}       \def\rt{\right}       \def\ae{\hb{a.e.}}
       \def\lan{\langle}     \def\as{\hb{a.s.}}
\def\ds{\displaystyle}  \def\ran{\rangle}     \def\tr{\hb{tr$\,$}}
     \def\llan{\lt\lan\!}  
\def\no{\noindent}      \def\rran{\!\rt\ran}  \def\var{\hb{var$\,$}}
\def\hp{\hphantom}         
\def\nn{\nonumber}         
\def\rf{\eqref}            
\def\cd{\cdot}             
\def\deq{\triangleq}    \def\({\Big (}        \def\ba{\begin{aligned}}
\def\les{\leqslant}     \def\){\Big )}        \def\ea{\end{aligned}}
\def\ges{\geqslant}     \def\[{\Big[}         \def\bel{\begin{equation}\label}
\def\ti{\tilde}         \def\]{\Big]}         \def\ee{\end{equation}}
     \def\q{\quad}         
\def\h{\hat}            \def\qq{\qquad}       
                                              
\def\a{\alpha}  \def\G{\Gamma}      \def\Om{\Omega}  
   \def\D{\Delta}   \def\d{\delta}   \def\F{\Phi}     
   \def\Th{\Theta}    \def\Si{\Sigma}  
\def\f{\varphi}   \def\l{\lambda}        \def\e{\varepsilon}
    \def\i{\infty}      
\newtheoremstyle{thry}% name
{}      % Space above
{}      % Space below
{\sl}   % Body font
{}      % Indent amount
{\bf}   % Theorem head font
{.}     % Punctuation after theorem head
{.5em}  % Space after theorem head
{}      % Theorem head spec (can be left empty, meaning 'normal')
\theoremstyle{thry}

\newtheorem{theorem}{Theorem}[section]
\newtheorem{proposition}[theorem]{Proposition}

\newtheorem{lemma}[theorem]{Lemma}

\theoremstyle{definition}
\newtheorem{definition}[theorem]{Definition}
\newtheorem{example}[theorem]{Example}

\theoremstyle{remark}
\newtheorem{remark}[theorem]{Remark}

\def\punct{}
\newtheoremstyle{dotless}{}{}{\rm}{}{\bf}{\punct}{.5em}{}
\theoremstyle{dotless}

\newtheorem{taggedassumptionx}{}
\newenvironment{taggedassumption}[1]
 {\renewcommand\thetaggedassumptionx{#1}\taggedassumptionx}
 {\endtaggedassumptionx}

\makeatletter
   
   \@addtoreset{equation}{section}
   \newcommand{\setword}[2]{%
   \phantomsection
   #1\def\@currentlabel{\unexpanded{#1}}\label{#2}%
   }
\makeatother

\begin{document}

\title{\bf Mean-Field  Linear-Quadratic Stochastic Differential Games}
\author{Jingrui Sun\thanks{Department of Mathematics, Southern University of Science and Technology,
                           Shenzhen 518055, China (Email: {\tt sunjr@sustech.edu.cn}).
                           This author is supported by NSFC Grant 11901280, Guangdong Basic and Applied Basic Research Foundation 2021A1515010031, and SUSTech start-up funds Y01286128 and Y01286228.}~~~
       Hanxiao Wang\thanks{Corresponding author. Department of Mathematics, National University of Singapore,
                           Singapore 119076, Singapore (Email: {\tt mathxw@nus.edu.sg}).
                           This author is supported by Singapore MOE AcRF Grant R-146-000-271-112.}~~~
      Zhen Wu\thanks{School of Mathematics, Shandong University, Jinan 250100,  China (Email: {\tt wuzhen@sdu.edu.cn}).
       This author is supported by NSFC Grants 11831010, 61961160732
        and Shandong Provincial Natural Science Foundation ZR2019ZD42.}
        }

\maketitle

\no{\bf Abstract.}
The paper is concerned with two-person zero-sum mean-field linear-quadratic stochastic differential games
over finite horizons.
By a Hilbert space method, a necessary condition and a sufficient condition are derived for the existence
of an open-loop saddle point.
It is shown that under the sufficient condition, the associated two Riccati equations admit unique strongly
regular solutions, in terms of which the open-loop saddle point can be represented as a linear feedback
of the current state.
When the game only satisfies the necessary condition,
an approximate sequence is constructed by solving a family of Riccati equations and closed-loop systems.
The convergence of the approximate sequence turns out to be equivalent to the open-loop solvability of the game,
and the limit is exactly an open-loop saddle point, provided that the game is open-loop solvable.

\ms

\no{\bf Keywords.}
linear-quadratic differential game, mean-field stochastic differential equation,
two-person, zero-sum, open-loop saddle point, Riccati equation, closed-loop representation,
perturbation approach.

\ms

\no{\bf AMS subject classifications.} 91A15, 93E20, 49N10, 49N70.

\section{Introduction}\label{Sec:Intro}

Let $(\Om,\cF,\dbP)$ be a complete probability space on which a standard one-dimensional
Brownian motion $W=\{W(t);0\les t<\i\}$ is defined.
The augmented natural filtration of $W$ is denoted by $\dbF=\{\cF_t\}_{t\ges0}$.
Consider the following controlled linear mean-field stochastic differential equation
(MF-SDE, for short) on a finite horizon $[0,T]$:
\bel{state}\left\{\begin{aligned}
   dX(s) &=\big\{A(s)X(s)+ \bar A(s)\dbE[X(s)] +B_1(s)u_1(s)+\bar B_1(s)\dbE[u_1(s)] \\
         &~\hp{=} +B_2(s)u_2(s)+\bar B_2(s)\dbE[u_2(s)] \big\}ds+\big\{C(s)X(s)+\bar C(s)\dbE[X(s)]\\
         &~\hp{=} +D_1(s)u_1(s)+\bar D_1(s)\dbE[u_1(s)]+ D_2(s)u_2(s)+\bar D_2(s)\dbE[u_2(s)]\big\}dW(s), \\
    X(0) &= x,
\end{aligned}\right.\ee
where $A,\bar A,\,C,\bar C:[0,T]\to\dbR^{n\times n}$, $B_i,\bar B_i,\, D_i,\bar D_i:[0,T]\to\dbR^{n\times m_i}$
($i=1,2$), called the {\it coefficients} of the {\it state equation} \rf{state}, are given deterministic functions.
The solution $X$ of \rf{state} is called a {\it state process},
and $u_i$ ($i=1,2$), belonging to the space
$$ \cU_i = \bigg\{\f:[0,T]\times\Om\to\dbR^{m_i} \bigm| \f\hb{~is $\dbF$-progressively measurable},
~\dbE\int^T_0|\f(s)|^2ds<\i\bigg\},
$$
is called the {\it control process} of Player $i$.
To measure the performance of the controls $u_1$ and $u_2$, we introduce the following functional:
\begin{align}\label{cost}
J(x;u_1,u_2)
&= \dbE\Big\{\lan GX(T),X(T)\ran +\lan\bar G\dbE[X(T)],\dbE[X(T)]\ran  \nn\\
&~\hp{=} +\int_0^T\llan\begin{pmatrix} Q & S_1^\top & S_2^\top\\
                                     S_1 & R_{11}   & R_{12}  \\
                                     S_2 & R_{21}   & R_{22}  \end{pmatrix}
                       \begin{pmatrix}X \\ u_1 \\ u_2\end{pmatrix},
                       \begin{pmatrix}X \\ u_1 \\ u_2\end{pmatrix}\rran ds \nn\\
&~\hp{=} +\int_0^T\llan\begin{pmatrix}\bar Q & \bar S_1^\top & \bar S_2^\top\\
                                    \bar S_1 & \bar R_{11}   & \bar R_{12}  \\
                                    \bar S_2 & \bar R_{21}   & \bar R_{22}  \end{pmatrix}
                       \begin{pmatrix}\dbE[X] \\ \dbE [u_1] \\ \dbE[u_2]\end{pmatrix},
                       \begin{pmatrix}\dbE[X] \\ \dbE [u_1] \\ \dbE[u_2]\end{pmatrix}\rran ds \Bigg\},
\end{align}
where $G$ and $\bar G$ are $n\times n$ symmetric matrices; $Q,\bar Q:[0,T]\to\dbR^{n\times n}$,
$S_i,\bar S_i:[0,T]\to\dbR^{m_i\times n}$, and $R_{ij},\bar R_{ij}:[0,T]\to\dbR^{m_i\times m_j}$
($i,j=1,2$) are deterministic functions with $Q=Q^\top$, $\bar Q=\bar Q^\top$, $R_{ji}=R^\top_{ij}$
and $\bar R_{ji}=\bar R_{ij}^\top$ ($i,j=1,2$).
In the Lebesgue integral on the right-hand side of \rf{cost},
we have suppressed the argument $s$, and we will do so in the sequel as long as no ambiguity arises.

\ms

The functional $J(x;u_1,u_2)$ represents the cost of Player $1$ and the payoff of Player $2$ for
using $u_1$ and $u_2$ to control the state process that starts from $x$.
Naturally, in this two-person zero-sum mean-field linear-quadratic stochastic differential game
(Problem (MF-SG), for short), Player 1 wishes to minimize \rf{cost} by selecting his/her control
from $\cU_1$, and Player 2 wishes to maximize \rf{cost} by selecting his/her control from $\cU_2$.
The control pair $(u_1^*,u_2^*)$ acceptable to both players is called an {\it open-loop saddle point}
of Problem (MF-SG), which is mathematically defined by the following inequalities:
\bel{define-saddle-point}
J(x;u^*_1,u_2)\les J(x;u^*_1,u^*_2)\les J(x;u_1,u^*_2),\q\forall (u_1,u_2)\in\cU_1\times\cU_2.
\ee
From \rf{define-saddle-point} one sees that if one of the players keeps his/her control $u_i^*$ unchanged,
the other cannot benefit by changing his/her control.
In this sense, an open-loop saddle point (if exists) will be the best choice for both players.

\ms

When $\bar A$, $\bar B_i$, $\bar C$, $\bar D_i$, $\bar G$, $\bar Q$, $\bar S_i$ and $\bar R_{ij}$ ($i,j=1,2$)
all vanish,  Problem (MF-SG) reduces to the classical two-person zero-sum linear-quadratic (LQ, for short)
stochastic differential game (Problem (SG), for short), which has been studied for a long history
and is widely applied in engineering, economy, and biology, etc.
Since the purpose of the paper is not to make a lengthy survey on the literature, we only list here some
closely related works (see, e.g., \cite{Bernhard1979,Zhang2005,Mou-Yong2006,Delfour2007,Delfour-Sbarba2009,Sun-Yong2014,Yu2015})
and refer the reader to the book \cite{Sun-Yong20202} by Sun--Yong for more details and references cited therein.
It is particularly worthy to mention that in a recent paper \cite{Sun2020} by Sun, the {\it strongly regular}
solvability of the Riccati equation associated with Problem (SG) is established under the so-called
{\it uniform convexity-concavity condition}. This result brings new insights into the two-person zero-sum
LQ stochastic differential game and serves as a foundation for our study on Problem (MF-SG).

\ms

Mean-field stochastic optimal control problems, which can be regarded as special cases of
the mean-field stochastic differential game in the sense that one is interested in a single
decision maker, have also attracted a lot of attention; see, for example,
\cite{Ahmed2007,Andersson2011,Buckdahn2011,Pham2018,Brandis2012,Bensoussan2013}.
The mean-field LQ stochastic optimal control problem was initially studied by Yong \cite{Yong2013}
and was later generalized by Huang--Li--Yong \cite{Huang2014}, Yong \cite{Yong2017},
Sun \cite{Sun2017}, Li--Sun--Xiong \cite{Li-Sun-Xiong2019}, and Sun--Wang \cite{Sun-Wang2019}
to various cases.
Let us briefly recall the  motivation for studying mean-field LQ stochastic optimal control problems proposed by Yong \cite{Yong2013}.
In some cases, one hopes that the
optimal state process and/or control process could be not too sensitive with respect to the possible variation of the random events.
To achieve this, one needs to keep the variances var$[X]$ and var$[u]$ small.
Therefore, it is natural to take var$[X]$ and var$[u]$ into account and consider cost functionals of the form
\begin{align}\label{cost111}
J(x;u)
&= \dbE\Big\{\lan GX(T),X(T)\ran +g \var [X(T)] +\int_0^T \big[ \lan Q(s)X(s),X(s)\ran \nn\\
&~\hp{=} + q(s) \var [X(s)]+\lan R(s)u(s),u(s)\ran+r(s)\var[u(s)]  \big] ds \Big\}.
\end{align}
In particular, in the mean-variance model (see \cite{Zhou-Li2000}, for example), the cost functional is simply
$\mu\var[X(T)]-\dbE[X(T)]$.
Note that
$$
\var[X(s)]=\dbE[|X(s)|^2]-(\dbE[X(s)])^2, \q \var[u(s)]=\dbE[|u(s)|^2]-(\dbE[u(s)])^2.
$$
The control problem with functional \rf{cost111} is actually a mean-field LQ control problem,
due to the presence of $(\dbE[X(s)])^2$ and $(\dbE[u(s)])^2$.
Another motivation for studying such type of problems is that the mean-field SDEs, also called McKean--Vlasov SDEs,
 can be used to describe particle systems at the mesoscopic level.
Recently, the mean-field/McKean--Vlasov SDE has been wildly used in mean-field game theory.
In general, it is the mean-square limit of the system of interacting particles
(see \cite{Andersson2011,Buckdahn2011}, for example).
For more details of such type of motivations,
we refer the reader to Huang--Malham\'{e}--Caines \cite{Huang2006}, Lasry--Lions \cite{Lasry2007},
 Bensoussan--Frehse--Yam \cite{Bensoussan2013},
Carmona--Delarue \cite{Carmona1}, and the references cited therein.
Along with the development of mean-field LQ optimal control problems,
the LQ differential games for mean-field SDEs have also attracted extensive research,
among which, we would like to mention Bensoussan--Sung--Yam--Yung \cite{Bensoussan2016},
Graber \cite{Graber2016}, Barreiro-Gomez--Duncan--Tembine \cite{Barreiro2019},
Li--Shi--Yong \cite{Li-Shi-Yong2020}, Moon \cite{Moon2020}, and Tian--Yu--Zhang \cite{Ran-Yu-Zhang2020}.

\ms

For mean-field LQ control problems, two Riccati equations are derived by Yong \cite{Yong2013}
to construct an open-loop optimal control.
The solvability of these two Riccati equations is established in \cite{Yong2013} under
certain positivity conditions and is further shown to be equivalent to the uniform convexity
of the cost functional by Sun \cite{Sun2017}.
However, to our best knowledge, there are few significant results on the general solvability
of the Riccati equations associated with Problem (MF-SG) so far.
One of the main contributions of this paper is to fill up this gap.
Compared with Sun \cite{Sun2020}, the additional difficulty mainly comes from the solvability of
the second Riccati equation associated with Problem (MF-SG). To overcome this difficulty,
besides establishing a technical \autoref{lemma-algebra},
we also show that the solution of the first Riccati equation satisfies a {\it comparison property}
(i.e., inequality \rf{Thm:solvability-RE1-main1} in \autoref{Thm:solvability-RE1}).
This property follows from the following observations:
One control in the saddle point of Problem (SG)
is optimal for a backward stochastic LQ control problem
and the value function of this backward problem is given exactly in terms of the solution to the first Riccati equation.
This observation is interesting in its own right and to our best knowledge, it is completely new in the literature.
\ms

In the literature, forward-backward stochastic differential equations (FBSDEs, for short) are usually used to
characterize the open-loop solvability of Problem (MF-SG) (see, for example, \cite{Sun-Yong2014,Sun-Yong20202}).
This method is suitable for deciding whether a control pair is an open-loop saddle point or not,
but not very effective in constructing open-loop saddle points,
because  the associated Riccati equations might be not solvable and then
the optimality system cannot be decoupled.
This paper provides an alternative characterization (see \autoref{thm:perturbation-approach})
for the open-loop solvability of Problem (MF-SG) by a perturbation approach, which can be regarded
as another important contribution.
It is worthy to point out that the characterization is new even for Problem (SG) (in which there are
no mean-field terms present).

\ms

The idea is to add two terms, $\e\|u_{1}\|^2$ and $-\e\|u_2\|^2$, to the original functional so that
the Problem (MF-SG) with the new functional $J_\e(x;u_1,u_2)\deq J(x;u_1,u_2)+\e\|u_{1}\|^2-\e\|u_{2}\|^2$
admits a unique open-loop point $(u_1^\e,u_2^\e)$ that can be represented explicitly in terms of
the solutions to the associated Riccati equations.
Then using the boundedness/convergence of the family $\{(u_1^\e,u_2^\e)\}_{\e>0}$ to justify the open-loop solvability
of the original game.
The main difficulty here is that the value function $V_\e(x)$ of the perturbed game is not monotone in $\e$,
due to which the technique used in the LQ control problem (see \cite{Sun-Li-Yong2016}) cannot be applied directly.
The significant difference between the perturbation methods of controls and games is illustrated
by presenting an elaborate example (see \autoref{exap-perturbation1}).
To overcome the difficulty, we restudy the perturbation approach by a  Hilbert space method,
which helps us to change the boundedness problem of  $\{(u_1^\e,u_2^\e)\}_{\e>0}$ into an equivalent one:
the norm estimate of some perturbed  operators with special structures
(see \autoref{Prop:operator}).
Furthermore, it is found that the explicit upper bound estimate \rf{bound} in \autoref{Prop:operator}
also plays a crucial role in proving  the strong convergence of $\{(u_1^\e,u_2^\e)\}_{\e>0}$,
because in \autoref{thm:perturbation-approach} we hope to show that $\{(u_1^\e,u_2^\e)\}_{\e>0}$ itself is strongly convergent
when Problem (MF-SG) is open-loop solvable.

\ms

To summarize, we list the main contributions of the paper as follows.

\ms

(1) The open-loop solvability of Problem (MF-SG) is studied by a Hilbert space method.
A necessary and sufficient condition for the existence of an open-loop saddle point
is derived (see \autoref{Thm:suffi-nece-condition}).

\ms

(2) Under the uniform convexity-concavity condition, the strongly regular solvability
of the Riccati equations associated with Problem (MF-SG) is established (see \autoref{Thm:solvability-RE1}
and \autoref{Thm:solvability-RE2}).
Further, in terms of the solutions to the Riccati equations, a closed-loop representation of the unique
open-loop saddle point is obtained (see \autoref{open-loop-sovability}).

\ms

(3) Under a necessary condition for the existence of an open-loop saddle point,
an equivalent characterization of the open-loop solvability is established
by a perturbation approach.
This approach also provides an explicit procedure for finding open-loop saddle
points (see \autoref{thm:perturbation-approach}).

\ms

In other words, under the uniform convexity-concavity condition, we first extend the results obtained in Sun \cite{Sun2020}
to the mean-field system.
As explained before, to prove the solvability of  the associated Riccati equations,
we need to make some new observations and to overcome some new difficulties.
Then under the weaker convexity-concavity condition,
we develop a perturbation approach to characterize the open-loop solvability of Problem (MF-SG).
This approach is first established for the game problem and can be regarded as the most technical part in the paper.
%because the value function $V_\e(x)$ is not monotone in $\e$.

\ms

The rest of the paper is organized as follows.
Section \ref{Sec:Pre} collects some preliminary results.
Section \ref{Sec:Rep-Functional} is devoted to the study of the performance functional
from a Hilbert space point of view.
Section \ref{Sec:open-loop-Riccati} establishes the solvability of the associated Riccati
equations and provides a closed-loop representation of the open-loop saddle point.
Section \ref{sec:perturbation} investigates the open-loop solvability of Problem (MF-SG)
by a perturbation method.
An example is presented in Section \ref{Sec:example} to illustrate the results obtained
in previous sections.

\section{Preliminaries}\label{Sec:Pre}
Throughout this paper,  let $\dbR^{n\times m}$ be the Euclidean space consisting of $n\times m$ real matrices,
endowed with the Frobenius inner product $\lan M,N\ran\deq\tr[M^\top N]$,
where $M^\top$ and $\tr(M)$ stand for the transpose  and  the trace of  $M$, respectively.
The norm of a matrix $M$ induced by the Frobenius inner is denoted by $|M|$ and the identity matrix of size $n$ is denoted by $I_n$.
Let $\dbS^n$  be the subspace of $\dbR^{n\times n}$ consisting of symmetric matrices
and $\dbS^n_+$ be the subset of $\dbS^n$ consisting of positive semidefinite matrices.
For any Euclidean space $\dbH$ (which could be $\dbR^n$, $\dbR^{n\times m}$, $\dbS^n$, etc.),
we introduce the following spaces:
\begin{align*}
C([0,T];\dbH):
   &\hb{~~the space of $\dbH$-valued, continuous functions on $[0,T]$}; \\
L^\i(0,T;\dbH):
   &\hb{~~the space of $\dbH$-valued,  essentially bounded functions  on $[0,T]$};\\
L^2_{\mathcal{F}_T}(\Om;\dbH):
   &\hb{~~the space of $\mathcal{F}_T$-measurable, $\dbH$-valued random variables $\xi$} \\
   &\hb{~~such that $\dbE|\xi|^2<\i$}; \\
L_\dbF^2(0,T;\dbH):
   &\hb{~~the space of $\dbF$-progressively measurable, $\dbH$-valued processes} \\
   &\hb{~~$\f:[0,T]\times\Om\to\dbH$ with $\ds\dbE\int_0^T|\f(s)|^2ds<\i$};\\
L_\dbF^2(\Om;C([0,T];\dbH)):
   &\hb{~~the space of $\dbF$-adapted, continuous, $\dbH$-valued processes} \\
   &\hb{~~$\f:[0,T]\times\Om\to\dbH$ with $\dbE\[\ds\sup_{s\in[0,T]}|\f(s)|^2\]<\i$}.
\end{align*}
We denote the norm of the Banach space $\cX$ by $\|\cd\|_{\cX}$,
which is often simply written as $\|\cd\|$ when no confusion occurs.
For $M,N\in\dbS^n$, we use the notation $M\ges N$ (respectively, $M>N$)
to indicate that $M-N$ is positive semidefinite (respectively, positive definite).
For any $\dbS^n$-valued measurable function $F$ on $[0,T]$, we denote
$$\left\{\begin{aligned}
  F\ges 0\q &\Longleftrightarrow\q F(s)\ges 0,\q \ae~s\in[0,T];\\
 F> 0\q &\Longleftrightarrow\q F(s)> 0,\q \ae~s\in[0,T];\\
 F\gg 0\q &\Longleftrightarrow\q F(s)\ges \d I_n,\q \ae~s\in[0,T],~\hbox{for some } \d>0.
\end{aligned}\right. $$
For self-adjoint linear operators $\cM$ and $\cN$ defined on the Hilbert space $\cH$,
we call $\cM$ a {\it positive operator} if  $\lan\cM x,x\ran\ges 0;\, \forall x\in\cH$ (see \cite[page 317, Definition 2]{Yosida}),
and we use $\cM\ges\cN$ to indicate that $\cM-\cN$ is a positive operator.

\ms
To guarantee that Problem (MF-SG) is well-posed,
we impose the following assumptions for the state equation \rf{state} and the functional \rf{cost}.

\begin{taggedassumption}{(H1)}\label{ass:A1}
The coefficients of state equation \rf{state} satisfy
$$
A,\bar A, C,\bar C \in L^\i(0,T;\dbR^{n\times n});
\q  B_i,\bar B_i, D_i,\bar D_i \in L^\i(0,T;\dbR^{n\times m_i}),~ i=1,2.
$$
\end{taggedassumption}

\begin{taggedassumption}{(H2)}\label{ass:A2}
The weighting matrices in the quadratic functional \rf{cost} satisfy:
$G,\bar G\in\dbS^n$, $Q,\bar Q\in L^\i(0,T;\dbS^{n})$, and for $i,j=1,2$,
$$
S_i, \bar S_i\in L^\i(0,T;\dbR^{m_i\times n}),
\q R_{ij}, \bar R_{ij}\in L^\i(0,T;\dbR^{m_i\times m_j}),\q R_{ij}^\top=R_{ji},\q \bar R_{ij}^\top=\bar R_{ji}.
$$
\end{taggedassumption}

Under \ref{ass:A1}, by \cite[Proposition 2.1]{Yong2013},
state equation \rf{state} admits a unique solution $X\in L^2_{\dbF}(\Om;C([0,T];\dbR^n))$.
If the assumption \ref{ass:A2} also holds,
the random variables on the right-hand side of  \rf{cost} are integrable and Problem (MF-SG) is well-posed.
Now we recall two important notions of LQ game problems.

\begin{definition}\label{definition-saddle-point}
A control pair $(u_1^*,u_2^*)\in\cU_1\times\cU_2$ is called an {\it open-loop saddle point}
of Problem (MF-SG) for the initial state $x\in\dbR^n$ if
\bel{definition-saddle-points}
J(x;u^*_1,u_2)\les J(x;u^*_1,u^*_2)\les J(x;u_1,u^*_2),\q\forall (u_1,u_2)\in\cU_1\times\cU_2.
\ee
Problem (MF-SG) is said to be {\it open-loop solvable} at $x$,
if it has an open-loop saddle point for $x$.
\end{definition}

\begin{definition}\label{definition-value-function}
For any $x\in\dbR^n$,  $V(x)$ is called a {\it value} of Problem (MF-SG) at $x$ if
\bel{definition-VF}
V(x)=\inf_{u_1\in\cU_1}\sup_{u_2\in\cU_2}J(x;u_1,u_2)
=\sup_{u_2\in\cU_2}\inf_{u_1\in\cU_1}J(x;u_1,u_2).
\ee
\end{definition}

Note that the value function $V$ is well-defined at $x\in\dbR^n$
only when the second equality in \rf{definition-VF} holds.
If $(u_1^*,u_2^*)\in\cU_1\times\cU_2$ is an open-loop saddle point for $x$, then
\bel{VF-OP}
V(x)=J(x;u_1^*,u_2^*).
\ee
In the following, let us make some preparations for the subsequent analysis of our main results.
We first present a primary lemma, which seems to be new and is crucial to proving the solvability of the Riccati equations.
The proof is sketched in Appendix for completeness.

\begin{lemma}\label{lemma-algebra}
Let $M\in\dbS^n_+$, $K\in\dbR^{n\times m}$ and $L\in\dbR^{n\times n}$. Then for any $\d>0$,
\bel{lemma-algebra-main}
L^\top ML-L^\top M K(K^\top M K+\d I_m)^{-1}K^\top ML\ges 0.
\ee
\end{lemma}

The following result is concerned with bounded linear operators,
by which we shall develop a perturbation approach for the open-loop solvability of Problem (MF-SG) in Section \ref{sec:perturbation}.
Let $\cH_1$ and $\cH_2$ be two real Hilbert spaces. Let
$$ \cM_{ij}:\cH_j\to\cH_i,\q i,j=1,2 $$
be linear bounded operators with $\cM_{ji}=\cM_{ij}^*$,
where $\cM_{ij}^*$ denotes the adjoint operator of $\cM_{ij}$. Then
$$ \cM \deq \begin{pmatrix}\cM_{11} & \cM_{12} \\ \cM_{21} & \cM_{22}\end{pmatrix}$$
is a self-adjoint linear bounded operator on the product Hilbert space $\cH_1\times\cH_2$
equipped with the inner product
$$\llan\begin{pmatrix}x_1 \\ y_1\end{pmatrix},\begin{pmatrix}x_2 \\ y_2\end{pmatrix}\rran
\deq \lan x_1,x_2\ran + \lan y_1,y_2\ran, \q\forall x_1,x_2\in\cH_1,~ y_1,y_2\in\cH_2. $$

\begin{proposition}\label{Prop:operator}
Suppose that $\cM_{11}$ and $-\cM_{22}$ are positive  operators; that is $\cM_{11}\ges 0$ and $-\cM_{22}\ges 0$.
Then for any $\e>0$,
$$ \cM_\e \deq \begin{pmatrix}\cM_{11}+\e I & \cM_{12} \\ \cM_{21} & \cM_{22}-\e I\end{pmatrix}$$
is invertible. Moreover,
\begin{align}\label{bound}
\|\cM_\e^{-1}\cM\| \les 1, \q\forall \e>0.
\end{align}
\end{proposition}

\begin{proof}
Since $\cM_{11}+\e I\ges\e I$, $\cM_{11,\e} \deq \cM_{11}+\e I$ is invertible with $\|\cM_{11,\e}^{-1}\|\les\e^{-1}$.
Similarly, $\cM_{22,\e} \deq \cM_{22}-\e I$ is invertible with $\|\cM_{22,\e}^{-1}\|\les\e^{-1}$,
and the self-adjoint operator
$$ \F_\e \deq \cM_{22,\e}-\cM_{21}\cM_{11,\e}^{-1}\cM_{12}\equiv\cM_{22,\e}-\cM_{12}^*\cM_{11,\e}^{-1}\cM_{12} $$
is invertible with $\|\F_{\e}^{-1}\|\les\e^{-1}$. Now it is straightforward to verify that $\cM_\e$ is invertible with inverse
$$ \cM_\e^{-1} =
\begin{pmatrix}\cM_{11,\e}^{-1}+(\cM_{11,\e}^{-1}\cM_{12})\F_\e^{-1}(\cM_{11,\e}^{-1}\cM_{12})^* & -(\cM_{11,\e}^{-1}\cM_{12})\F_\e^{-1} \\[1mm]
               -\F_\e^{-1}(\cM_{11,\e}^{-1}\cM_{12})^* & \F_\e^{-1} \end{pmatrix}. $$
To prove \rf{bound}, we write
$$\cM_\e^{-1}\cM
= \begin{pmatrix}I & 0 \\ 0 & I\end{pmatrix} - \e\cM_\e^{-1}\begin{pmatrix*}[r]I & 0 \\ 0 & -I\end{pmatrix*}
= \begin{pmatrix}I & 0 \\ 0 & I\end{pmatrix} - \e\begin{pmatrix}\cM_{11,\e} & \cM_{12} \\ -\cM_{12}^* & -\cM_{22,\e}\end{pmatrix}^{-1}. $$
Denote
\bel{def-B-e}
\h\cM_\e\deq\begin{pmatrix}\cM_{11,\e} & \cM_{12} \\ -\cM_{12}^* & -\cM_{22,\e}\end{pmatrix}.
\ee
Then
\begin{align}
\cM_\e^{-1}\cM=I-\e\h\cM_\e^{-1},
\end{align}
and thus
\begin{align}
(\cM_\e^{-1}\cM)^*\cM_\e^{-1}\cM
&= (I-\e\h\cM_\e^{-1} )^*(I-\e\h\cM_\e^{-1} )\nn \\
&= I-\e (\h\cM_\e^{-1})^*-\e\h\cM_\e^{-1}+\e^2 (\h\cM_\e^{-1})^* \h\cM_\e^{-1}\nn\\
\label{A-A-star}&=I-\e(\h\cM_\e^{-1})^*\big[\h\cM_\e+\h\cM_\e^*-\e I\big]\h\cM_\e^{-1}.
\end{align}
Note that
\begin{align*}
\h\cM_\e+\h\cM_\e^*-\e I
&=\begin{pmatrix}\cM_{11,\e} & \cM_{12} \\ -\cM_{12}^* & -\cM_{22,\e}\end{pmatrix}
+\begin{pmatrix}\cM_{11,\e} & -\cM_{12} \\ \cM_{12}^* & -\cM_{22,\e}\end{pmatrix}
-\e\begin{pmatrix}I & 0 \\ 0 & I\end{pmatrix}\\
&=\begin{pmatrix}2\cM_{11}+\e I & 0 \\ 0 & -2\cM_{22}+\e I\end{pmatrix}.
\end{align*}
Using the fact that $\cM_{11}$ and $-\cM_{22}$ are positive  operators, we have
\begin{align*}
\begin{pmatrix}2\cM_{11}+\e I & 0 \\ 0 & -2\cM_{22}+\e I\end{pmatrix}\ges 0;
\end{align*}
that is
$$
\h\cM_\e+\h\cM_\e^*-\e I\ges 0,
$$
which implies that
$$
\e(\h\cM_\e^{-1})^*\big[\h\cM_\e+\h\cM_\e^*-\e I\big]\h\cM_\e^{-1}\ges 0.
$$
Combining the above with \rf{A-A-star} yields
\begin{align}
0\les(\cM_\e^{-1}\cM)^*\cM_\e^{-1}\cM=I-\e(\h\cM_\e^{-1})^*\big[\h\cM_\e+\h\cM_\e^*-\e I\big]\h\cM_\e^{-1}\les I.
\end{align}
Thus, we get
$$
\|\cM_\e^{-1}\cM\|^2=\|(\cM_\e^{-1}\cM)^*\cM_\e^{-1}\cM\|\les 1.
$$
The proof is complete.
\end{proof}

\section{Representation of the functional}\label{Sec:Rep-Functional}
In this section, we shall study the functional \rf{cost} from a Hilbert space point of view and
represent it as a quadratic functional of the controls $(u_1,u_2)$,
by which  a necessary condition and a sufficient condition  will be derived for the open-loop solvability.
For any $u_i\in \cU_i$ ($i=1,2$), consider the following MF-SDE:
\bel{state-ui}\left\{\begin{aligned}
   dX^{i,0}(s) &=\big\{A(s)X^{i,0}(s)+ \bar A(s)\dbE[X^{i,0}(s)] +B_i(s)u_i(s)+\bar B_i(s)\dbE[u_i(s)] \big\}ds \\
         &\hp{= }~ +\big\{C(s)X^{i,0}(s)+\bar C(s)\dbE[X^{i,0}(s)] + D_i(s)u_i(s)\\
         &\hp{= }~+\bar D_i(s)\dbE[u_i(s)]\big\}dW(s),\q s\in[0,T],\\
     X^{i,0}(0)&= 0.
\end{aligned}\right.\ee
Under \ref{ass:A1}, the above MF-SDE admits a unique solution $X^{i,0}\in L_\dbF^2(\Om;C([0,T];\dbR^n))$
satisfying
\bel{K}
\dbE\[\sup_{s\in[0,T]}|X^{i,0}(s)|^2\]\les K\dbE\int_0^T|u_i(s)|^2ds,
\ee
where the constant $K>0$ is independent of $u_i$.
Thus we can define two bounded linear operators $\cL_i:\cU_i\to L_\dbF^2(\Om;C([0,T];\dbR^n))$
and $\h\cL_i:\cU_i\to L_{\cF_T}^2(\Om;\dbR^n)$ as follows:
\bel{def-cL-cM}
\cL_i u_i=X^{i,0},\q \h\cL_i u_i=X^{i,0}(T), \q \forall u_i\in\cU_i;\q i=1,2.
\ee
Also we can define the linear operators $\cN:\dbR^n\to L_\dbF^2(\Om;C([0,T];\dbR^n))$
and $\h\cN:\dbR^n\to L_{\cF_T}^2(\Om;\dbR^n)$ as follows:
\bel{def-cN-cO}
\cN x=X^{0,x},\q \h\cN x=X^{0,x}(T),\q \forall x\in\dbR^n,
\ee
with $X^{0,x}$ being the unique solution to the following MF-SDE:
\bel{state-0-x}\left\{\begin{aligned}
   dX^{0,x}(s) &=\big\{A(s)X^{0,x}(s)+ \bar A(s)\dbE[X^{0,x}(s)]  \big\}ds \\
         &\hp{= }~ +\big\{C(s)X^{0,x}(s)+\bar C(s)\dbE[X^{0,x}(s)] \big\}dW(s),\q s\in[0,T],\\
     X^{0,x}(0)&= x.
\end{aligned}\right.\ee
For any given $(x,u_1,u_2)\in\dbR^n\times\cU_1\times\cU_2$,
it is easily checked that $X^{0,x}+X^{1,0}+X^{2,0}$ satisfies state equation \rf{state}.
Thus, by the uniqueness of the solution to MF-SDE \rf{state}, we have
\bel{X-Xi}
X=X^{0,x}+X^{1,0}+X^{2,0}=\cN x+\cL_1 u_1+\cL_2 u_2,\q \forall(x,u_1,u_2)\in\dbR^n\times\cU_1\times\cU_2.
\ee
In particular, the terminal value of $X$ can be represented by
\begin{align}
X(T)&=X^{0,x}(T)+X^{1,0}(T)+X^{2,0}(T)\nn\\
\label{X-Xi-T}&=\h\cN x+\h\cL_1 u_1+\h\cL_2 u_2,\q \forall(x,u_1,u_2)\in\dbR^n\times\cU_1\times\cU_2.
\end{align}
Then using \rf{X-Xi}--\rf{X-Xi-T}, by the completion
of squares technique it is straightforward to obtain the  following representation of the functional \rf{cost}:
\begin{align}
 J(x;u_1,u_2)
=\lan\cM u,u\ran+2\lan \cK x,u\ran+\lan\cO x,x\ran,\nn\\
\label{cost-functional-rewrite}~ \forall x\in\dbR^n,~ u=(u_1^\top,u^\top_2)^\top\in\cU_1\times\cU_2,
\end{align}
where
\bel{def-M}
\cM\deq \begin{pmatrix}\cM_{11}& \cM_{12}\\\cM_{21} & \cM_{22}\end{pmatrix},~
\cK\deq \begin{pmatrix}\cK_{1}\\\cK_{2} \end{pmatrix},~
\cO\deq \h\cN^*(G+\dbE^*\bar G\dbE )\h\cN+\cN^*(Q+\dbE^*\bar Q\dbE )\cN,
\ee
with
\begin{align}
\nn&\cM_{ij}\deq \h\cL_j^*(G+\dbE^*\bar G\dbE )\h\cL_i+\cL_j^*(Q+\dbE^*\bar Q\dbE )\cL_i+R_{ji}+\dbE^*\bar R_{ji}\dbE\\
\label{def-Mij}&\qq\q+(S_j+\dbE^*\bar S_j\dbE)\cL_i+\cL_j^*(S_i^\top+\dbE^*\bar S_i^\top\dbE),\q i,j=1,2;\\
%
%\nn&\cM_{12}\deq \h\cL_2^*(G+\dbE^*\bar G\dbE )\h\cL_1+\cL_2^*(Q+\dbE^*\bar Q\dbE )\cL_1+R_{12}+\dbE^*\bar R_{12}\dbE\\
%
%\nn&\qq\q+(S_2+\dbE^*\bar S_2\dbE)\cL_1+\cL_2^*(S_1^\top+\dbE^*\bar S_1^\top\dbE),\q \cM_{21}\deq\cM_{12}^*; \\
%
\label{def-k}&\cK_{i}\deq \h\cL_i^*(G+\dbE^*\bar G\dbE )\h\cN+\cL_i^*(Q+\dbE^*\bar Q\dbE )\cN+(S_i+\dbE^*\bar S_i\dbE)\cN,\q i=1,2.
\end{align}
By the above expression of $\cM_{ij};i,j=1,2$, we have $\cM_{ji}=\cM_{ij}^*;\,i,j=1,2$,
which implies that $\cM$ is a self-adjoint operator. With the representation \rf{cost-functional-rewrite},
we provide the following characterization for the open-loop saddle points of Problem (MF-MG).

\begin{proposition}\label{Thm:suffi-nece-condition}
Let {\rm\ref{ass:A1}--\ref{ass:A2}} hold. Let $x\in\dbR^n$ be any given initial state
and $u^*=(u_1^{*\top},u_2^{*\top})^\top\in\cU_1\times\cU_2$.
Then $u^*$ is an open-loop saddle point of {\rm Problem (MF-MG)} for $x$ if and only if
\bel{main-suffi-nece-condition}
(-1)^{i+1} \cM_{ii}\ges 0;~i=1,2\q\hbox{and}\q \cM u^*+\cK x=0.
\ee
\end{proposition}

\begin{proof}
By \autoref{definition-saddle-point}, $u^*=(u_1^{*\top},u_2^{*\top})^\top$ is an open-loop saddle point of {\rm Problem (MF-MG)}
if and only if
\begin{align}
\label{suffi-nece-condition1}
J(x;u_1^*+\l u_1,u_2^*)-J(x;u_1^*,u_2^*)\ges 0,\q \forall u_1\in\cU_1,~\l\in\dbR; \\
\label{suffi-nece-condition2}
J(x;u_1^*,u_2^*+\l u_2)-J(x;u_1^*,u_2^*)\les 0,\q \forall u_2\in\cU_2,~\l\in\dbR.
\end{align}
For any $u_1\in\cU_1$ and $\l\in\dbR$, by \rf{cost-functional-rewrite} we have
\begin{align}
\nn&J(x;u_1^*+\l u_1,u_2^*)-J(x;u_1^*,u_2^*)\\
&\q=\l^2\big\lan\cM_{11}u_1,u_1\big\ran+2\l\big[ \big\lan\cM_{11}u^*_1,u_1\big\ran
+\big\lan\cM_{12}u^*_2,u_1\big\ran+\big\lan\cK_{1}x,u_1\big\ran\big].
\end{align}
Thus \rf{suffi-nece-condition1}  holds if and only if
\bel{suffi-nece-condition3}
\cM_{11}\ges 0\q\hbox{and}\q \cM_{11}u^*_1+\cM_{12}u^*_2+\cK_{1}x=0.
\ee
By the same argument as the above, we can show that \rf{suffi-nece-condition2}  holds if and only if
\bel{suffi-nece-condition4}
\cM_{22}\les 0\q\hbox{and}\q \cM_{22}u^*_2+\cM_{21}u^*_1+\cK_{2}x=0.
\ee
Note that \rf{main-suffi-nece-condition} is equivalent to \rf{suffi-nece-condition3} and \rf{suffi-nece-condition4}.
The proof is thus complete.
\end{proof}

From \autoref{Thm:suffi-nece-condition}, we see that the following {\it convexity-concavity condition},
\bel{convex-condition}
\begin{aligned}
\lan\cM_{11}u_1,u_1\ran &=J(0;u_1,0)\ges 0,\q \forall u_1\in\cU_1;\\
\lan\cM_{22}u_2,u_2\ran &=J(0;0,u_2)\les 0,\q \forall u_2\in\cU_2,
\end{aligned}
\ee
is  necessary for the existence of an open-loop saddle point.
Next we introduce a condition slightly stronger than \rf{convex-condition}:

\begin{taggedassumption}{(H3)}\label{ass:H3}\rm
There exists a constant $\a>0$ such that
\bel{uniform-convex-condition}
\begin{aligned}
\lan\cM_{11}u_1,u_1\ran &=J(0;u_1,0)\ges \a \|u_1\|^2,\q~~ \forall u_1\in\cU_1;\\
\lan\cM_{22}u_2,u_2\ran &=J(0;0,u_2)\les -\a \|u_2\|^2,\q \forall u_2\in\cU_2.
\end{aligned}
\ee
\end{taggedassumption}
If \ref{ass:H3}  holds, for convenience we usually write \rf{uniform-convex-condition} as follows:
\begin{align*}
\lan\cM_{11}u_1,u_1\ran &=J(0;u_1,0)\gg 0,\q \forall u_1\in\cU_1;\\
\lan\cM_{22}u_2,u_2\ran &=J(0;0,u_2)\ll 0,\q \forall u_2\in\cU_2.
\end{align*}
The following result shows that the {\it uniform convexity-concavity condition} \ref{ass:H3}
is  sufficient for the open-loop solvability of Problem (MF-MG).

\begin{proposition}\label{suffi-condition}
Let {\rm \ref{ass:A1}--\ref{ass:H3}} hold, and the operators $\cM$ and $\cK$ be defined by \rf{def-M}.
Then $\cM$ is invertible and for any $x\in\dbR^n$,
{\rm Problem (MF-MG)} admits a unique open-loop saddle point $u^*=(u_1^{*\top},u_2^{*\top})^\top$ given by
\bel{main-suffi-condition}
u^*=-\cM^{-1}\cK x.
\ee
\end{proposition}

\begin{proof}
By \rf{uniform-convex-condition}, we obtain that the operators $\cM_{11}$, $\cM_{22}$ and
$\F\deq \cM_{22}-\cM_{21}\cM_{11}^{-1}\cM_{12}$ are  invertible.
Then it is straightforward to verify that $\cM$ is invertible with inverse
\bel{M-inverse}
\cM^{-1}=\begin{pmatrix}
\cM_{11}^{-1}+(\cM_{11}^{-1}\cM_{12})\F^{-1}(\cM_{11}^{-1}\cM_{12})^* & -(\cM_{11}^{-1}\cM_{12})\F^{-1} \\[1mm]
-\F^{-1}(\cM_{11}^{-1}\cM_{12})^* & \F^{-1}
\end{pmatrix}.
\ee
The remaining results follow from \autoref{Thm:suffi-nece-condition} directly.
\end{proof}

\section{Open-loop saddle points and Riccati equations}\label{Sec:open-loop-Riccati}
According to \autoref{suffi-condition}, under the uniform convexity-concavity condition \ref{ass:H3},
the open-loop saddle point for the given  $x\in\dbR^n$ can be uniquely determined by \rf{main-suffi-condition}.
However, since $\cM^{-1}\cK$ is an abstract operator and very complicated,
it is usually difficult to find the open-loop saddle point by computing \rf{main-suffi-condition} directly.
Thus in this section, we shall give a more explicit form of  the open-loop saddle point
by introducing two associated Riccati equations. Furthermore,
it will be shown that the unique open-loop saddle point admits a closed-loop representation.

\ms

Recall from \cite{Sun2020} that the Riccati equation associated with Problem (SG)  reads
\bel{Ric1}\left\{\begin{aligned}
  & \dot P+PA+A^\top P+C^\top PC+Q \\
  & \hp{\dot P} -(PB+C^\top PD+S^\top)(R+D^\top PD)^{-1}(B^\top P+D^\top PC+ S)=0, \\
  & P(T)=G,
\end{aligned}\right.\ee
where
\begin{align}
\nn&B=(B_1,B_2),\q  D=(D_1,D_2),
\q S= \begin{pmatrix} S_{1}\\ S_{2}\end{pmatrix},
\q R= \begin{pmatrix} R_{11} & R_{12}\\ R_{21} & R_{22}\end{pmatrix},\\
\nn& R+D^\top PD= \begin{pmatrix} R_{11}+D_1^\top PD_1 & R_{12}+D_1^\top PD_2 \\
R_{21}+D_2^\top PD_1 & R_{22}+D_2^\top PD_2\end{pmatrix},\\
\label{def-B}&B^\top P+D^\top PC+ S= \begin{pmatrix} B_1^\top P+D_1^\top PC+ S_1\\
B_2^\top P+D_2^\top PC+ S_2\end{pmatrix}.
\end{align}

\begin{definition}\label{definition-SP-solution}\rm
An absolutely continuous function $P:[0,T]\to \dbS^n$ is called a  {\it strongly regular solution} of Riccati equation \rf{Ric1} if
\begin{enumerate}[(i)]
\item For $i=1,2$, $(-1)^{i+1}[R_{ii}+D_i^\top PD_i]\gg 0$, and

\item  $P$ satisfies \rf{Ric1}  almost everywhere on $[0,T]$.

\end{enumerate}
\end{definition}

To establish the solvability of Riccati equation \rf{Ric1}, we introduce the following two optimal control problems:
For $i=1,2$, consider the state equation
\bel{state-i-control}\left\{\begin{aligned}
   dX(s) &=\big\{A(s)X(s)+ \bar A(s)\dbE[X(s)] +B_i(s)u_i(s)+\bar B_i(s)\dbE[u_i(s)]\big\}ds \\
         &\hp{= ~} +\big\{C(s)X(s)+\bar C(s)\dbE[X(s)] + D_i(s)u_i(s)+\bar D_i(s)\dbE[u_i(s)]\big\}dW(s),\\
     X(0)&= x,
\end{aligned}\right.\ee
and the cost functional
\begin{align}\label{cost-i-control}
\nn &J_i(x;u_i)= (-1)^{i+1}\dbE\big\{\lan GX(T),X(T)\ran +\lan\bar G\dbE[X(T)],\dbE[X(T)]\ran  \\
&\qq +\int_0^T\Bigg[\llan\begin{pmatrix}Q & S_i^\top \\
                                              S_i & R_{ii}\end{pmatrix}
                                 \begin{pmatrix}X \\ u_i \end{pmatrix},
                                 \begin{pmatrix}X \\ u_i \end{pmatrix}\rran +\llan\begin{pmatrix}\bar Q & \bar S_i^\top \\
                                              \bar S_i &\bar R_{ii} \end{pmatrix}
                                 \begin{pmatrix}\dbE[X] \\ \dbE [u_i] \end{pmatrix},
                                 \begin{pmatrix}\dbE[X] \\ \dbE [u_i] \end{pmatrix}\rran \Bigg]ds \Bigg\}.
\end{align}
If the mean-field terms in the above vanish, then \rf{state-i-control} and \rf{cost-i-control} reduce to
\bel{state-i-control-no}\left\{\begin{aligned}
   dX(s) &=\big\{A(s)X(s) +B_i(s)u_i(s)\big\}ds +\big\{C(s)X(s) + D_i(s)u_i(s)\big\}dW(s),\\
     X(0)&= x,
\end{aligned}\right.\ee
and
\begin{align}\label{cost-i-control-no}
\nn \cJ_i(x;u_i)
&= (-1)^{i+1}\dbE\Bigg\{\lan GX(T),X(T)\ran
 +\int_0^T\llan\begin{pmatrix}Q & S_i^\top \\
                                              S_i & R_{ii} \end{pmatrix}
                                 \begin{pmatrix}X \\ u_i \end{pmatrix},
                                 \begin{pmatrix}X \\ u_i\end{pmatrix}\rran ds  \Bigg\}.
\end{align}
The Riccati equations associated with the above LQ control problems read
\bel{Ric-i}\left\{\begin{aligned}
  & \dot P_i+P_iA+A^\top P_i+C^\top P_iC+Q -(P_iB_i+C^\top P_iD_i+S_i^\top)\\
  & \hp{\dot P}\q \times(R_{ii}+D_i^\top P_iD_i)^{-1}(B_i^\top P_i+D_i^\top P_{i}C+ S_i)=0,\q i=1,2, \\
  & P_i(T)=G,\q i=1,2.
\end{aligned}\right.\ee

\begin{theorem}\label{Thm:solvability-RE1}
Let {\rm \ref{ass:A1}--\ref{ass:H3}} hold.
Then Riccati equation \rf{Ric1} admits a unique strongly regular solution $P\in C([0,T];\dbS^n)$.
Moreover, the strongly regular solution $P$ satisfies
\bel{Thm:solvability-RE1-main1}
P_1(t)\les P(t)\les P_2(t),\q\forall t\in[0,T],
\ee
and
\bel{Thm:solvability-RE1-main2}
(-1)^{i+1}[R_{ii}+\bar R_{ii}+(D_i+\bar D_i)^\top P(D_i+\bar D_i)]\gg 0;\q i=1,2,
\ee
where $P_i\in C([0,T];\dbS^n)$ ($i=1,2$) is the unique solution of \rf{Ric-i}.
\end{theorem}

\begin{proof}
Note that the assumption \ref{ass:H3} implies that for any $(u_1,u_2)\in\cU_1\times\cU_2$,
$$
J_1(0;u_1)=J(0;u_1,0)\gg 0 \q \hbox{and}\q J_2(0;u_2)=-J(0;0,u_2)\gg 0.
$$
Thus for $i=1,2$, we obtain from \cite[Theorems 4.2 and 4.4]{Sun2017} that Riccati equation \rf{Ric-i}
admits a unique solution $P_i\in C([0,T];\dbS^n)$ satisfying
\bel{Thm:solvability-RE1-11}
(-1)^{i+1}[R_{ii}+D_i^\top P_iD_i]\gg 0,
\ee
and
\bel{Thm:solvability-RE1-1}
(-1)^{i+1}[R_{ii}+\bar R_{ii}+(D_i+\bar D_i)^\top P_i(D_i+\bar D_i)]\gg 0.
\ee
According to \cite[Theorem 5.2]{Sun2017},  \rf{Thm:solvability-RE1-11} implies that
the mapping $u_i\mapsto\cJ_i(0;u_i)$ is uniformly convex.
Thus,
\begin{align}
\nn&\cJ(0;u_1,0)=\cJ_1(0;u_1)\gg 0,\q~~ \forall u_1\in\cU_1;\\
\label{cJ-convex}& \cJ(0;0,u_2)=-\cJ_2(0;u_2)\ll 0,\q \forall u_2\in\cU_2,
\end{align}
where  $\cJ$ denotes the functional of Problem (SG);
that is $\cJ(x;u_1,u_2)\deq J(x;u_1,u_2)$
with $\bar A$, $\bar B_i$, $\bar C$, $\bar D_i$, $\bar G$, $\bar Q$, $\bar S_i$ and $\bar R_{ij}$ ($i,j=1,2$)
all vanishing.
Then from \cite[Theorem 4.3]{Sun2020}, we obtain that Riccati equation \rf{Ric1} has a strongly regular solution
$P\in C([0,T];\dbS^n)$.
% satisfying \rf{Thm:solvability-RE1-main1}.
%Combining \rf{Thm:solvability-RE1-main1} with \rf{Thm:solvability-RE1-1},
%we get \rf{Thm:solvability-RE1-main2} immediately.
%\begin{align}
%&(-1)^{i+1}[R_{ii}+D_i^\top P D_i]\ges(-1)^{i+1}[R_{ii}+D_i^\top P_iD_i]\gg 0,\\
%
%\nn&(-1)^{i+1}[R_{ii}+\bar R_{ii}+(D_i+\bar D_i)^\top P(D_i+\bar D_i)]\\
%
%&\q\ges(-1)^{i+1}[R_{ii}+\bar R_{ii}+(D_i+\bar D_i)^\top P_i(D_i+\bar D_i)]\gg 0.
%\end{align}

\ms
To prove the uniqueness, we suppose that $\bar P, \ti P\in C([0,T];\dbS^n)$
are two strongly regular solutions of \rf{Ric1}.
Then $\bar P$ and $\ti P$ satisfy:
$$
(-1)^{i+1}[R_{ii}+D_i^\top\ti PD_i]\gg 0\q \hbox{and}\q (-1)^{i+1}[R_{ii}+D_i^\top \bar PD_i]\gg 0,\q i=1,2.
$$
Similar to \rf{M-inverse},  $R+D^\top\bar PD$ and $R+D^\top\ti PD$ are invertible
with their inverses being bounded. Denote $\D P=\bar P-\ti P$.
Then $\D P$ satisfies the following linear ordinary differential equation:
\bel{Delta-Ric1}\left\{\begin{aligned}
  & \D\dot P+\D PA+A^\top\D P+C^\top \D PC -(\D PB+C^\top\D PD)(R+D^\top\bar PD)^{-1} \\
  & \hp{\dot P}\times(B^\top \bar P+D^\top\bar PC+ S)-(\ti PB^\top +C^\top\ti PD+ S^\top)(R+D^\top\bar PD)^{-1} D^\top\D PD\\
  & \hp{\dot P}\times(R+D^\top\ti PD)^{-1}(B^\top \bar P+D^\top\bar PC+ S)-(\ti PB^\top +C^\top\ti PD+ S^\top)\\
  &\hp{\dot P} \times(R+D^\top\ti PD)^{-1}(B^\top \D P+D^\top\D PC)=0,\\
  & \D P(T)=0.
\end{aligned}\right.\ee
Note that $\bar P$, $\ti P$, $(R+D^\top\bar PD)^{-1}$ and $(R+D^\top\ti PD)^{-1}$ are bounded.
Then by a standard argument using the Gr\"{o}nwall's inequality, we get $\D P\equiv 0$,
which yields the uniqueness of the strongly regular solution to \rf{Ric1}.

\ms
Next let us prove the unique strongly regular solution $P\in C([0,T];\dbS^n)$ of \rf{Ric1} satisfies \rf{Thm:solvability-RE1-main1}.
To this end, for any fixed $u_2\in\cU_2$, we introduce  the following  LQ control problem:  Consider the
state equation
\bel{state-u-2}\left\{\begin{aligned}
   dX(s) &=\big\{A(s)X(s) +B_1(s)u_1(s)+ B_2(s)u_2(s) \big\}ds\\
         &~\hp{=} +\big\{C(s)X(s)+D_1(s)u_1(s)+ D_2(s)u_2(s)\big\}dW(s), \\
    X(0) &= x,
\end{aligned}\right.\ee
and the cost functional
\begin{align}\label{cost-u-2}
\cJ_{u_2}(x;u_1)
%= \dbE\Bigg\{\lan GX(T),X(T)\ran
%%
% +\int_0^T\llan\begin{pmatrix} Q & S_1^\top & S_2^\top\\
 %                                    S_1 & R_{11}   & R_{12}  \\
 %                                   S_2 & R_{21}   & R_{22}  \end{pmatrix}
 %                     \begin{pmatrix}X \\ u_1 \\ u^*_2\end{pmatrix},
 %                     \begin{pmatrix}X \\ u_1 \\ u^*_2\end{pmatrix}\rran ds \Bigg\}\nn\\
                       %
&\deq \cJ(x;u_1,u_2)= \dbE\bigg\{\lan GX(T),X(T)\ran
 +\int_0^T\[\lan QX,X\ran+2\lan S_1X,u_1\ran \nn\\
&\qq\q+\lan R_{11}u_1,u_1\ran +2\lan X, S_2^\top u_2\ran+2\lan u_1,R_{12} u_2\ran+\lan R_{22} u_2, u_2\ran\] ds \bigg\}.
\end{align}
Note that \rf{cJ-convex} implies the mapping $u_1\mapsto \cJ_{u_2}(x;u_1)=\cJ(x;u_1,u_2)$ is uniformly convex,
then by \cite[Theorem 4.3]{Sun-Li-Yong2016} the unique optimal control $\bar u_1(\cd)\equiv\bar u_1(\cd;u_2)$ of the above LQ control problem admits the following closed-loop representation:
\bel{closed-loop-repre-X}
\bar u_1(s)=\Th(s) \bar X(s)+v(s)\equiv\Th(s) \bar X(s;u_2)+v(s;u_2),  \q s\in[0,T],
\ee
where %
      $$\begin{aligned}
        \Th &= -\big(R_{11}+D_1^\top P_1D_1\big)^{-1}\big(B_1^\top P_1+D_1^\top P_1C+S_1\big),\\
          v &= -\big(R_{11}+D_1^\top P_1D_1\big)^{-1}\big(B_1^\top Y+D_1^\top Z+D_1^\top P_1 D_2u_2+R_{12}u_2\big);
      \end{aligned}$$
 $P_1$ is uniquely determined by Riccati equation \rf{Ric-i};
$(Y(\cd),Z(\cd))\equiv(Y(\cd;u_2),Z(\cd;u_2))$ solves the backward stochastic differential  equation:
\bel{eta-beta}\left\{\begin{aligned}
         dY(s) &=-\big[(A+B_1\Th)^\top Y(s) +(C+D_1\Th)^\top Z(s) +(C+D_1\Th)^\top P_1D_2u_2\\
                  &\hp{=-\big[} +\Th^{\top} R_{12}u_2 +P_1B_2u_2 +S_2^\top u_2 \big]ds +Z(s) dW(s), \q s\in[0,T],\\
          Y(T) &=0;
\end{aligned}\right.\ee
and $\bar X(\cd)\equiv\bar X(\cd;u_2)$ is the solution of the closed-loop system:
\bel{state-u2-star}\left\{\begin{aligned}
   d\bar X(s) &=\big\{A(s)\bar X(s) +B_1(s)[\Th(s)\bar X(s)+v(s)]+ B_2(s)u_2(s) \big\}ds\\
         &~\hp{=} +\big\{C(s)\bar X(s)+D_1(s)[\Th(s)\bar X(s)+v(s)]+ D_2(s)u_2(s)\big\}dW(s), \\
    \bar X(0) &= x.
\end{aligned}\right.\ee
Moreover,
\begin{align}\label{cost-u2-star}
&\cJ_{u_2}(x;\bar u_1)=\cJ_{u_2}(x;\Th \bar X+v)=\cJ(x;\Th \bar X+v,u_2)\nn\\
&= \dbE\bigg\{\lan P_1(0)x,x\ran+2\lan Y(0),x\ran
 +\int_0^T\[\lan P_1D_2u_2,D_2u_2\ran+2\lan Y,B_2u_2\ran  \nn\\
&\qq +2\lan Z,D_2u_2\ran+\big\lan(R_{11}+D_1^\top P_1D_1)^{-1}(B_1^\top Y+D_1^\top Z+D_1^\top P_1D_2u_2+R_{12}u_2),\nn\\
&\qq\qq (B_1^\top Y+D_1^\top Z+D_1^\top P_1D_2u_2+R_{12}u_2)\big\ran+\lan R_{22} u_2, u_2\ran\] ds \bigg\}.
\end{align}
 Since the condition \rf{cJ-convex} holds,  by \cite[Theorem 4.4]{Sun2020}
Problem (SG) admits a unique open-loop saddle point $(u_1^*,u_2^*)$; that is
$$
\cJ(x;u^*_1,u_2^*)=\sup_{u_2\in\cU_2}\inf_{u_1\in\cU_1}\cJ(x;u_1,u_2)=\inf_{u_1\in\cU_1}\sup_{u_2\in\cU_2}\cJ(x;u_1,u_2)=\lan P(0)x,x\ran,
$$
 which implies that
\bel{cJ-line}
\inf_{u_1\in\cU_1}\cJ_{u^*_2}(x;u_1)=\cJ_{u^*_2}(x;u^*_1) \q\hbox{and}\q \cJ(x;u^*_1,u_2^*)=\sup_{u_2\in\cU_2}\cJ(x;u^*_1,u_2).
\ee
Recall that the mapping $u_1\mapsto \cJ(x;u_1,u_2^*)$ is uniformly convex,
which implies that $u_1^*$ is the unique  control satisfying the first equality in the above,
thus we have
\bel{closed1}
\cJ(x;u^*_1,u_2^*)=\cJ_{u_2^*}(x;u^*_1)=\cJ_{u^*_2}(x;\Th \bar X^*+v^*)=\cJ(x;\Th \bar X^*+v^*,u^*_2),
\ee
with $\bar X^*(\cd)=\bar X(\cd;u_2^*)$ and $v^*(\cd)=v(\cd;u_2^*)$.
On the other hand, by the second equality in \rf{cJ-line}
and the closed-loop representation \rf{closed-loop-repre-X} of the optimal control $\bar u_1$,
we get
$$
\cJ(x;u^*_1,u_2^*)\ges \cJ(x;u_1^*,u_2)=\cJ_{u_2}(x;u_1^*)\ges \cJ_{u_2}(x;\bar u_1)= \cJ(x;\Th \bar X+v,u_2),\q\forall u_2\in\cU_2,
$$
with $\bar X(\cd)=\bar X(\cd;u_2)$ and $v(\cd)=v(\cd;u_2)$. Combining the above with \rf{closed1} yields that
\bel{u2-star1}
\cJ(x;\Th \bar X^*+v^*,u^*_2)\ges \cJ(x;\Th \bar X+v,u_2)\equiv \cJ(x;\Th \bar X(\cd;u_2)+v(\cd;u_2),u_2),\q\forall u_2\in\cU_2.
\ee
Taking $u_2=0$ in \rf{u2-star1} and then making use of \rf{cost-u2-star} and \rf{eta-beta}, we have
$$
\lan P(0)x,x\ran=\cJ(x;\Th \bar X^*+v^*,u^*_2)\ges \cJ(x;\Th \bar X+v,0)=\lan P_1(0)x,x\ran.
$$
%It follows that
%$$
%\lan P(0)x,x\ran)\ges\lan P_1(0)x,x\ran.
%$$
Using the above arguments to the Problem (SG) with the initial pair replaced by $(t,x)\in[0,T]\times\dbR^n$, we obtain
$$
\lan P(t)x,x\ran\ges\lan P_1(t)x,x\ran,\q\forall(t,x)\in[0,T]\times\dbR^n.
$$
In a similar manner, we can also show that
$$
\lan P(t)x,x\ran\les\lan P_2(t)x,x\ran,\q\forall(t,x)\in[0,T]\times\dbR^n.
$$
Thus, the unique strongly regular solution $P$ of Riccati equation \rf{Ric1} satisfies \rf{Thm:solvability-RE1-main1}.
Finally, combining \rf{Thm:solvability-RE1-main1} with \rf{Thm:solvability-RE1-1},
we get \rf{Thm:solvability-RE1-main2} immediately. The proof is thus complete.
\end{proof}

%\begin{remark}
%The  {\it comparison property} \rf{Thm:solvability-RE1-main1} of the solutions $P$ and $P_i;i=1,2$
%can follow from \rf{cJ-convex} and \cite[Proposition 4.7]{Sun2020} directly.
%However, the arguments employed in the proof of \cite[Proposition 4.7]{Sun2020} get involved with some tedious calculations.
%Thus a new method is presented in our proof of \autoref{Thm:solvability-RE1}.
%\end{remark}

\begin{remark}
If we define the state processes $(Y,Z)$ by  \rf{eta-beta}
and the cost functional $\h J(x;u_2)\deq-\cJ(x;\Th \bar X+v,u_2)$ by \rf{cost-u2-star},
then the corresponding control problem is a backward LQ problem.
From \rf{u2-star1}, we see that under \ref{ass:H3},
if $(u_1^*,u_2^*)$ is the unique open-loop saddle point of Problem (SG),
then $u_2^*$ is optimal for the above designed backward problem.
For more results of backward LQ  control problems,
we refer the reader to \cite{Lim-Zhou2001,Li-Sun-Xiong2019,Sun-Wang20191} and the references cited therein.
\end{remark}

With the strongly regular solution $P$ of \rf{Ric1},
we now introduce the following deterministic two-person zero-sum LQ differential game problem
(Problem (DG), for short): Consider the state equation
\bel{state-deterministic}\left\{\begin{aligned}
   \dot{y}(s) &=[A(s)+ \bar A(s)]y(s) +[B_1(s)+\bar B_1(s)]v_1(s)+[B_2(s)+\bar B_2(s)]v_2(s),\\
     y(0)&= x,
\end{aligned}\right.\ee
and the functional
\begin{align}\label{cost-derterministic}
 &\bar J(x;v_1,v_2) = \big\lan (G+\bar G)y(T),y(T)\big\ran
 +\int_0^T\llan\begin{pmatrix} \Upsilon& \Gamma_1^\top & \Gamma_2^\top\\
                                              \G_1 & \bar\Sigma_{11} & \bar\Sigma_{12}\\
                                              \G_2 & \bar\Sigma_{21} & \bar\Sigma_{22}\end{pmatrix}
                                 \begin{pmatrix}y \\ v_1 \\ v_2\end{pmatrix},
                                 \begin{pmatrix}y \\ v_1 \\ v_2\end{pmatrix}\rran ds,
\end{align}
where
\bel{def-G}\left\{\begin{aligned}
&\Upsilon=Q+\bar Q+(C+\bar C)^\top P(C+\bar C),\\
&\G_i=(D_i+\bar D_i)^\top P(C+\bar C)+(S_i+\bar S_i),\q i=1,2,\\
&\bar\Sigma_{ij}=R_{ij}+\bar R_{ij}+(D_i+\bar D_i)^\top P(D_j+\bar D_j),\q i,j=1,2.
\end{aligned}\right.
\ee
Since the strongly regular solution $P$ also satisfies \rf{Thm:solvability-RE1-main2},
the matrices $\Si$ and $\bar\Si$, defined by
\bel{def-Si-barSi}
 \Sigma\equiv R+D^\top P D,\q\bar\Sigma\equiv R+\bar R+(D+\bar D)^\top P(D+\bar D),
\ee
are invertible.
The Riccati equation associated with Problem (DG) is
\bel{Ric2}\left\{\begin{aligned}
& \dot{\Pi}+\Pi(A+\bar A)+(A+\bar A)^\top\Pi+Q+\bar Q+(C+\bar C)^\top P(C+\bar C)-\big[\Pi(B+\bar B)\\
& \hp{\dot{\Pi}}\qq +(C+\bar C)^\top P(D+\bar D)+(S+\bar S)^\top\big]\big[R+\bar R+(D+\bar D)^\top P(D+\bar D)\big]^{-1}\\
& \hp{\dot \Pi}\q \times\big[(B+\bar B)^\top\Pi+(D+\bar D)^\top P(C+\bar C)+(S+\bar S)\big]=0, \\
& \Pi(T)=G+\bar G,
\end{aligned}\right.\ee
where $\bar B$, $\bar D$ and $\bar R$ are defined in a similar way to \rf{def-B}.
The following result shows that under \ref{ass:H3}, Riccati equation \rf{Ric2} is also solvable.

\begin{theorem}\label{Thm:solvability-RE2}
Let {\rm \ref{ass:A1}--\ref{ass:H3}} hold. Then Riccati equation \rf{Ric2} is uniquely solvable.
\end{theorem}

\begin{proof}
By \cite[Theorem 4.3]{Sun2020}, to prove the solvability of Riccati equation \rf{Ric2}, it suffices to show that
\bel{v1-v2-con}
\bar J(0;v_1,0)\gg0 \q\hbox{and}\q  \bar J(0;0,v_2)\ll0.
\ee
Denote
\begin{align}\label{cost-derterministic-i}
 &\bar J_1(x;v_1)= \lan (G+\bar G)y(T),y(T)\ran +\int_0^T\llan\begin{pmatrix} \bar\Upsilon_1& \bar\Gamma_1^\top \\
                                              \bar\G_1 & \bar\Sigma_{11} & \end{pmatrix}
                                 \begin{pmatrix}y \\ v_1\end{pmatrix},
                                 \begin{pmatrix}y \\ v_1\end{pmatrix}\rran ds,
\end{align}
where $y$ is the unique solution of \rf{state-deterministic} with $v_2\equiv0$, and
\bel{def-G-i}\left\{\begin{aligned}
&\bar\Upsilon_1=Q+\bar Q+(C+\bar C)^\top P_1(C+\bar C),\\
&\bar\G_1=(D_1+\bar D_1)^\top P_1(C+\bar C)+(S_1+\bar S_1),\\
&\bar\Sigma_{11}=R_{11}+\bar R_{11}+(D_1+\bar D_1)^\top P_1(D_1+\bar D_1),
\end{aligned}\right.
\ee
with $P_1$ being the unique solution of \rf{Ric-i} for $i=1$.
Recall the definition \rf{cost-i-control} of $J_1$  and note that $J_1(0;u_1)=J(0;u_1,0)\gg 0$.
By \cite[Theorem 4.4]{Sun2017}, we have
\bel{bar-J-convex}
\bar J_1(0;v_1)\gg 0.
\ee
On the other hand, for any $\d>0$, by \rf{cost-derterministic}--\rf{cost-derterministic-i} we obtain
\begin{align}
\nn&\q\bar J(0;v_1,0)-\bar J_1(0;v_1)+\d\|v_1\|^2\\
\nn&\q=\int_0^T\Big\{\big\lan(C+\bar C)^\top (P-P_1)(C+\bar C)y,y\big\ran+2\big\lan(D_1+\bar D_1)^\top(P- P_1)(C+\bar C)y,v_1\big\ran\\
\nn&\qq\qq +\big\lan[(D_1+\bar D_1)^\top(P- P_1)(D_1+\bar D_1)+\d I_{m_1}]v_1,v_1\big\ran\Big\}ds\\
\label{def-Q-R-S}&\q\equiv\int_0^T\big[\lan\cQ y,y\ran+2\lan\cS y,x\ran+\lan\cR v_1,v_1\ran\big]ds.
\end{align}
Note that $P\ges P_1$ (recalling \rf{Thm:solvability-RE1-main1}), thus
$$
\cR\equiv(D_1+\bar D_1)^\top(P- P_1)(D_1+\bar D_1)+\d I_{m_1}\ges \d I_{m_1}\gg 0.
$$
Moreover, by \autoref{lemma-algebra} we have
\begin{align}
\nn&\cQ-\cS^\top\cR^{-1}\cS=(C+\bar C)^\top (P-P_1)(C+\bar C)-(C+\bar C)^\top(P- P_1)(D_1+\bar D_1)\\
&\qq\times[(D_1+\bar D_1)^\top(P- P_1)(D_1+\bar D_1)+\d I_{m_1}]^{-1}(D_1+\bar D_1)^\top(P- P_1)(C+\bar C)\ges0.\nn
\end{align}
Thus, the weighting matrices $\cQ$, $\cS$ and $\cR$ satisfy the so-called {\it standard condition}
in the literature (see \cite[Chapter 6]{Yong-Zhou1999}, for example) of  LQ optimal control problems.
Then
\begin{align}
\nn&\bar J(0;v_1,0)-\bar J_1(0;v_1)+\d\|v_1\|^2\\
\nn&\q=\int_0^T\[\big\lan(\cQ-\cS^\top\cR^{-1}\cS) y,y\big\ran+\big\lan\cR(v_1+\cR^{-1}\cS y),(v_1+\cR^{-1}\cS y)\big\ran\]ds\\
\label{bar-J-J1+d}&\q\ges 0.
\end{align}
Since $\d>0$ is arbitrary, we get
\bel{bar-J-J1}
\bar J(0;v_1,0)-\bar J_1(0;v_1)\ges 0,
\ee
which, together with \rf{bar-J-convex}, implies that
\bel{bar-J-v1}
\bar J(0;v_1,0)\gg 0.
\ee
Similarly, we can also prove that
\bel{bar-J-v2}
\bar J(0;0,v_1)\ll 0.
\ee
Combining \rf{bar-J-v1} with \rf{bar-J-v2}, we get \rf{v1-v2-con},
which completes the proof of the existence.
Then by the same arguments as  in the proof of \autoref{Thm:solvability-RE1},
we obtain the unique solvability of Riccati equation \rf{Ric2}.
\end{proof}

\begin{remark}
It is noteworthy that the  comparison property \rf{Thm:solvability-RE1-main1} of  $P$ and $P_i;i=1,2$
serves as a crucial bridge to prove the solvability of Riccati equation \rf{Ric2} in \autoref{Thm:solvability-RE2}.
The technical \autoref{lemma-algebra} is used to show the weighting matrices $\cQ$, $\cS$ and $\cR$ defined by
\rf{def-Q-R-S} exactly satisfy the so-called standard condition.
\end{remark}

With the strongly regular solvability of  Riccati equations \rf{Ric1}--\rf{Ric2} having been established,
we present the closed-loop representation for the open-loop saddle point of Problem (MF-SG).

\begin{theorem}\label{open-loop-sovability}
Let {\rm\ref{ass:A1}--\ref{ass:H3}} hold.
Let $P\in C([0,T];\dbS^n)$ be the strongly regular solution to Riccati equation \rf{Ric1}
satisfying \rf{Thm:solvability-RE1-main1}--\rf{Thm:solvability-RE1-main2}
and $\Pi\in C([0,T];\dbS^n)$ be the solution to Riccati equation \rf{Ric2}.
Then with the notations
\begin{align}
\label{def-Si}&\Sigma= R+D^\top P D,\q\bar\Sigma\equiv R+\bar R+(D+\bar D)^\top P(D+\bar D),\\
\label{def-TH}&\Th=-\Sigma^{-1}\big(B^\top P+D^\top PC+S\big),\\
\label{def-bar-Th}
&\bar\Th=-\bar\Sigma^{-1}[(B+\bar B)^\top\Pi+(D+\bar D)^\top P(C+\bar C)+S+\bar S],
\end{align}
the unique open-loop saddle point $u^*=(u_1^{*\top},u_2^{*\top})^\top$ for the initial state $x$ has the following closed-loop representation:
\bel{open-closed-re}
 u^* = \Th\big\{X^*-\dbE[X^*]\big\}+\bar\Th\dbE[X^*],
\ee
where $X^*$ is the solution to the closed-loop system:
\bel{closed-syst}\left\{\begin{aligned}
dX^*(s) &=\big\{(A+B\Th)(X^*-\dbE[X^*])+[(A+\bar A)+(B+\bar B)\bar\Th]\dbE[X^*]\big\}ds \\
  &\hp{=\ }+\big\{(C+D\Th)(X^*-\dbE[X^*])+[(C+\bar C)+(D+\bar D)\bar\Th]\dbE[X^*]\big\}dW(s),\\
   X^*(0) &=x.
\end{aligned}\right.\ee
Moreover,  the value function of {\rm Problem (MF-SG)} is given by $V(x)=\lan\Pi(0)x,x\ran$.
\end{theorem}

\begin{remark}
By \autoref{open-loop-sovability}, we give an explicit representation for the unique open-loop saddle point of Problem (MF-SG).
Indeed, the LQ problems occupied the center stage for research in  control theory not only for its elegant solutions
but also for its ability to approximate more general nonlinear problems,
as pointed out by Wang--Zariphopoulou--Zhou in their recent work \cite{Wang-Z2020}
of the stochastic control approach in reinforcement learning.
\end{remark}

In order to prove \autoref{open-loop-sovability}, we need the following lemma, whose proof is standard
and is similar to that of \cite[Theorem 4.1]{Sun-Yong2014}.

\begin{lemma}\label{lemma-optimality}
A pair $(\bar u_1,\bar u_2)\in\cU_1\times\cU_2$ is an open-loop saddle point of {\rm Problem (MF-SG)} if and only if
\rf{convex-condition} holds and $\bar u=(\bar u_1^{\top},\bar u_2^{\top})^\top$ satisfies
\begin{align}
\nn&B(s)^\top\bar Y(s)+ \bar B(s)^\top\dbE[\bar Y(s)]+D(s)^\top\bar Z(s)+\bar D(s)^\top\dbE[\bar Z(s)]+S(s)\bar X(s)\\
\label{opti-condition}
&+\bar S(s)\dbE[\bar X(s)]+R(s)\bar u(s)+\bar R(s)\dbE[\bar u(s)]=0,\q s\in[0,T],\q \as,
\end{align}
where $(\bar X,\bar Y,\bar Z)$ is the unique solution to the following
mean-field forward-backward SDE (MF-FBSDE, for short):
\bel{FBSDE}\left\{\begin{aligned}
d\bar X(s) &=\big\{A\bar X+ \bar A\dbE[\bar X] +B\bar u+\bar B\dbE[\bar u]\big\}ds \\
         &\hp{= }~+\big\{C\bar X+ \bar C\dbE[\bar X] +D\bar u+\bar D\dbE[\bar u]\big\}dW(s),\q s\in[0,T],\\
d\bar Y(s) &=-\big\{A^\top\bar Y+ \bar A^\top\dbE[\bar Y] +C^\top\bar Z+ \bar C^\top\dbE[\bar Z] +Q\bar X+\bar Q\dbE[\bar X]\\
&\qq+S^\top\bar u+\bar S^\top\dbE[\bar u]\big\}ds+\bar ZdW(s),\q s\in[0,T],\\
\bar X(0)&= x,\q \bar Y(T)=G\bar X(T)+\bar G\dbE[\bar X(T)].
\end{aligned}\right.\ee
\end{lemma}

\ms

\noindent
\emph{\textbf{Proof of \autoref{open-loop-sovability}.}}
We first prove that the control $u^*$ defined by \rf{open-closed-re} is the unique open-loop saddle point.
Denote
\begin{align}
Y^*&=P(X^*-\dbE[X^*])+\Pi\dbE[X^*],\nn\\
\label{denote-Y-star}
 Z^*&=P\big\{C(X^*-\dbE[X^*])+(C+\bar C)\dbE[X^*]+D(u^*-\dbE[u^*])+(D+\bar D)\dbE[u^*]\big\}.
\end{align}
Note that
\begin{align*}
&\dbE[Y^*]=\Pi\dbE[X^*],\q\dbE[u^*]=\bar\Th\dbE[X^*],\q\dbE[Z^*]=P\big\{(C+\bar C)\dbE[X^*]+(D+\bar D)\dbE[u^*]\big\}.
\end{align*}
By  It\^{o}'s formula, we have
\begin{align}
\nn& dY^*(t)=\dot{P}(X^*-\dbE[X^*])dt+P\big\{(A+B\Th)(X^*-\dbE[X^*])\big\}dt\\
&\qq+P\big\{(C+D\Th)(X^*-\dbE[X^*])+[(C+\bar C)+(D+\bar D)\bar\Th]\dbE[X^*]\big\}dW(t)\nn\\
&\qq+\dot{\Pi}\dbE[X^*]dt+\Pi[(A+\bar A)+(B+\bar B)\bar\Th]\dbE[X^*]dt\nn\\
%
%&\q=-(A^\top P+C^\top PC+Q-\Th^\top\Sigma\Th-PB\Th)(X^*-\dbE[X^*])dt\nn\\
%
%&\qq-\big[(A+\bar A)^\top\Pi+Q+\bar Q+(C+\bar C)^\top P(C+\bar C)-\bar\Th^\top\bar\Sigma^{-1}\bar\Th-\Pi(B+\bar B)\bar\Th\big]\dbE[X^*]dt\nn\\
%
%&\qq+P\big\{C(X^*-\dbE[X^*])+(C+\bar C)\dbE[X^*]+D\Th(X^*-\dbE[X^*])+\bar D\bar\Th\dbE[X^*]\big\}dW(t)\nn\\
%
%
&\q=-\big[A^\top P+C^\top PC+Q+(C^\top PD+S^\top)\Th\big](X^*-\dbE[X^*])dt-\big\{(A+\bar A)^\top\Pi+Q+\bar Q\nn\\
&\qq\q+(C+\bar C)^\top P(C+\bar C)+[(C+\bar C)^\top P(D+\bar D)+S^\top+\bar S^\top]\bar\Th\big\}\dbE[X^*]dt+Z^*dW(t)\nn\\
\nn&\q=-\Big\{ A^\top Y^*+\bar A^\top\dbE[ Y^*]+C^\top P\big\{ C(X^*-\dbE[X^*])+ (C+\bar C)\dbE[X^*]+D(u^*-\dbE[u^*])\\
&\qq\qq+(D+\bar D)\dbE[u^*]\big\}+\bar C^\top P\big\{ (C+\bar C)\dbE[X^*]+(D+\bar D)\dbE[u^*]\big\}+QX^*+\bar Q\dbE[X^*]\nn\\
&\qq\q+S^\top u^*+\bar S^\top\dbE[u^*]\Big\}dt+ Z^*dW(t)\nn\\
\nn&\q=-\big\{A^\top Y^*+ \bar A^\top\dbE[Y^*] +C^\top Z^*+ \bar C^\top\dbE[Z^*] +Q X^*+\bar Q\dbE[X^*]+S^\top u^*\\
\label{Y-star}&\qq\q+\bar S^\top\dbE[u^*]\big\}dt+Z^*dW(t).
\end{align}
Further, by the terminal values of $P$ and $\Pi$, we get
\begin{align}
\nn Y^*(T)&=P(T)\{X^*(T)-\dbE[X^*(T)]\}+\Pi(T)\dbE[X^*(T)]\\
\label{Y-star-T}
&=GX^*(T)+\bar G\dbE[X^*(T)].
\end{align}
%Notice that $X^*$ satisfies the (forward) SDE in \rf{FBSDE} with $\bar u=u^*$.
Thus $(X^*,Y^*,Z^*)$ satisfies MF-FBSDE \rf{FBSDE} with the control $\bar u\equiv u^*$.
Moreover, note that
\begin{align}
\nn& B^\top Y^*+\bar B^\top\dbE[Y^*]+D^\top Z^*+\bar D^\top\dbE[ Z^*]+SX^*+\bar S\dbE[X^*]+Ru^*+\bar R\dbE[ u^*]\\
%
%\label{opti-condition1}
&\nn\q=B^\top\{P(X^*-\dbE[X^*])+\Pi\dbE[X^*]\}+\bar B^\top\Pi\dbE[X^*]+SX^*+\bar S\dbE[X^*]+\bar R\bar\Th\dbE[X^*]\\
&\nn\qq+R\big\{\Th(X^*-\dbE[X^*])+\bar\Th\dbE[X^*]\big\}+\bar D^\top P\big\{(C+\bar C)\dbE[X^*]+(D+\bar D)\bar\Th\dbE[X^*]\big\}\\
&\nn\qq+D^\top P\big\{C(X^*-\dbE[X^*])+(C+\bar C)\dbE[X^*]+D\Th(X^*-\dbE[X^*])+(D+\bar D)\bar\Th\dbE[X^*]\big\}\\
&\nn\q=\big\{B^\top P+S+D^\top PC+(D^\top PD+R)\Th\big\}(X^*-\dbE[X^*])+\big\{(B+\bar B)^\top \Pi\\
&\nn\qq+S+\bar S+(D+\bar D)^\top P(C+\bar C)+[(D+\bar D)^\top P(D+\bar D)+R+\bar R]\bar\Th\big\}\dbE[X^*]\\
&\label{u-star}\q=0.
\end{align}
Then by \autoref{lemma-optimality},
the control $u^*$ defined by \rf{open-closed-re} is an open-loop saddle point.
The uniqueness of open-loop saddle points follows from \autoref{suffi-condition}.

\ms
By integration by parts and \rf{u-star}, we get
\begin{align}
\nn&\dbE\lan GX^*(T),X^*(T)\ran+\lan\bar G\dbE[X^*(T)],\dbE[X^*(T)]\ran
=\dbE\lan Y^*(T),X^*(T)\ran \\
\nn&\q=\dbE\lan Y^*(0),X^*(0)\ran+\dbE\int_0^T \[\big\lan B^\top Y^*+\bar B^\top\dbE[Y^*]+D^\top Z^*+\bar D^\top\dbE[ Z^*]-SX^*\\
\nn&\qq-\bar S\dbE[X^*],u^*\big\ran-\lan QX^*,X^*\ran-\lan\bar Q\dbE [X^*],X^*\ran\]ds\\
\nn&\q=\dbE\lan Y^*(0),X^*(0)\ran-\dbE\int_0^T \[\lan R u^*,u^*\ran+\lan\bar R \dbE[u^*],\dbE[u^*]\ran
+2\lan SX^*,u^*\ran\\
\nn&\qq+2\lan\bar S\dbE[X^*],\dbE[u^*]\ran+\lan QX^*,X^*\ran+\lan\bar Q\dbE [X^*],\dbE[X^*]\ran\]ds.
\end{align}
Noting that $Y^*(0)=\Pi(0)x$,  we have the following representation of the value function:
\begin{align}
\nn V(x)&=J(x;u^*)=\dbE\lan GX^*(T),X^*(T)\ran+\lan\bar G\dbE[X^*(T)],\dbE[X^*(T)]\ran\\
\nn&\q+\dbE\int_0^T \[\lan R u^*,u^*\ran+\lan\bar R \dbE[u^*],\dbE[u^*]\ran
+2\lan SX^*,u^*\ran+2\lan\bar S\dbE[X^*],\dbE[u^*]\ran\\
\nn&\q+\lan QX^*,X^*\ran+\lan\bar Q\dbE [X^*],\dbE[X^*]\ran\]ds\\
&=\dbE\lan Y^*(0),X^*(0)\ran=\lan \Pi(0)x,x\ran,\q x\in\dbR^n.\nn\qq\qq\qq\qq\qq\qq\hfill\qed
\end{align}

%\section{Two-Step scheme: a leader-follower approach}

\section{Open-loop solvability: a perturbation approach}\label{sec:perturbation}
%It is showed in \autoref{open-loop-sovability} that under the uniform convexity-concavity condition \ref{ass:H3},
%Problem (MF-SG) admits an unique open-loop saddle point,
%which can be further represented the linear feedback  \rf{open-closed-re} of the current state.
%In this section, we shall present a characterization for the open-loop solvability of Problem (MF-SG)
%without the assumption \ref{ass:H3}.
In \autoref{open-loop-sovability}, it is shown  that under the uniform convexity-concavity condition \ref{ass:H3},
Problem (MF-SG) is uniquely open-loop solvable and the unique open-loop saddle point
admits the closed-loop representation \rf{open-closed-re}.
In this section, we shall establish a characterization for the open-loop solvability of Problem (MF-SG)
without  assumption \ref{ass:H3}.
Recall from \autoref{Thm:suffi-nece-condition} that the following convexity-concavity condition
is necessary for the open-loop solvability of Problem (MF-SG):
\bel{convex-condition1}
\begin{aligned}
\lan\cM_{11}u_1,u_1\ran &=J(0;u_1,0)\ges 0,\q \forall u_1\in\cU_1;\\
\lan\cM_{22}u_2,u_2\ran &=J(0;0,u_2)\les 0,\q \forall u_2\in\cU_2.
\end{aligned}
\ee
Thus in this section, we always assume that \rf{convex-condition1} holds.

\subsection{The perturbation approach}\label{subsec:perturbation}
For each $\e>0$, we introduce the following perturbed  functional:
\begin{align}\label{cost-e}
\nn &J_\e(x;u_1,u_2)\deq J(x;u_1,u_2)+\e\dbE\int_0^T|u_1(s)|^2ds-\e\dbE\int_0^T|u_2(s)|^2ds\\
&\nn\q= \dbE\big\{\lan GX(T),X(T)\ran +\lan\bar G\dbE[X(T)],\dbE[X(T)]\ran  \\
&\hp{=\dbE\big\{ ~~} +\int_0^T\llan\begin{pmatrix}Q & S_1^\top & S_2^\top\\
                                              S_1 & R_{11}+\e I_{m_1} & R_{12}\\
                                              S_2 & R_{21} & R_{22}-\e I_{m_2}\end{pmatrix}
                                 \begin{pmatrix}X \\ u_1 \\ u_2\end{pmatrix},
                                 \begin{pmatrix}X \\ u_1 \\ u_2\end{pmatrix}\rran ds \nn\\
&\hp{=\dbE\big\{ ~~} +\int_0^T\llan\begin{pmatrix}\bar Q & \bar S_1^\top &\bar S_2^\top\\
                                              \bar S_1 &\bar R_{11} &\bar R_{12}\\
                                              \bar S_2 &\bar R_{21} &\bar R_{22}\end{pmatrix}
                                 \begin{pmatrix}\dbE[X] \\ \dbE [u_1] \\ \dbE[u_2]\end{pmatrix},
                                 \begin{pmatrix}\dbE[X] \\ \dbE [u_1] \\ \dbE[u_2]\end{pmatrix}\rran ds \Bigg\}.
\end{align}
We denote the two-person zero-sum LQ stochastic differential game associated with \rf{state}--\rf{cost-e} by Problem (MF-SG)$_\e$
and the value function  by $V_\e$.
Notice that
\bel{uniform-convex-condition-e}
\begin{aligned}
J_\e(0;u_1,0)&=J(0;u_1,0)+\e\dbE\int_0^T|u_1(s)|^2ds \ges \e \|u_1\|^2,\q~~ \forall u_1\in\cU_1;\\
J_\e(0;0,u_2)&=J(0;0,u_2)-\e\dbE\int_0^T|u_2(s)|^2ds \les -\e \|u_2\|^2,\q \forall u_2\in\cU_2.
\end{aligned}
\ee
Then with
\bel{R-e} R_\e\deq \begin{pmatrix} R_{11}+\e I_{m_1} & R_{12}\\
 R_{21} & R_{22}-\e I_{m_2}\end{pmatrix},
\ee
it follows from \autoref{Thm:solvability-RE1} and \autoref{Thm:solvability-RE2}  that the Riccati equations
\bel{Ric1-e}\left\{\begin{aligned}
  & \dot P_\e+P_\e A+A^\top P_\e+C^\top P_\e C+Q \\
  & \hp{\dot P} -(P_\e B+C^\top P_\e D+S^\top)(R_\e+D^\top P_\e D)^{-1}(B^\top P_\e+D^\top P_\e C+ S)=0, \\
  & P_\e(T)=G
\end{aligned}\right.\ee
and
\bel{Ric2-e}\left\{\begin{aligned}
& \dot{\Pi}_\e+\Pi_\e(A+\bar A)+(A+\bar A)^\top\Pi_\e+Q+\bar Q+(C+\bar C)^\top P_\e(C+\bar C)\\
& \hp{\dot{\Pi}} -\big[\Pi_\e(B+\bar B)+(C+\bar C)^\top P_\e(D+\bar D)+(S+\bar S)^\top\big]\\
& \hp{\dot \Pi} \times\big[R_\e+\bar R+(D+\bar D)^\top P_\e(D+\bar D)\big]^{-1}\\
& \hp{\dot \Pi} \times\big[(B+\bar B)^\top\Pi_\e+(D+\bar D)^\top P_\e(C+\bar C)+(S+\bar S)\big]=0, \\
& \Pi_\e(T)=G+\bar G
\end{aligned}\right.\ee
admit unique solutions $P_\e\in C([0,T];\dbS^n)$ and $\Pi_\e\in C([0,T];\dbS^n)$  satisfying
\begin{align}
&(-1)^{i+1}[R_{ii}+(-1)^{i+1}\e I_{m_i}+D_i^\top P_\e D_i]\gg 0,\q i=1,2,\\
\label{RE1-e-convex}
&(-1)^{i+1}[R_{ii}+\bar R_{ii}+(-1)^{i+1}\e I_{m_i}+(D_i+\bar D_i)^\top P_\e (D_i+\bar D_i)]\gg 0,\q i=1,2.
\end{align}
Denote
\begin{align}
\label{def-Si-e}&\Sigma_\e= R_\e+D^\top P_\e D,\q\bar\Sigma_\e= R_\e+\bar R+(D+\bar D)^\top P_\e(D+\bar D),\\
\label{def-TH-e}&\Th_\e=-\Sigma_\e^{-1}\big(B^\top P_\e+D^\top P_\e C+S\big),\\
\label{def-bar-Th-e}
&\bar\Th_\e=-\bar\Sigma^{-1}_\e[(B+\bar B)^\top\Pi_\e+(D+\bar D)^\top P_\e(C+\bar C)+S+\bar S].
\end{align}
By \autoref{open-loop-sovability},
the unique open-loop saddle point $u_\e=(u_{1}^{\e\top},u_{2}^{\e\top})^\top$ of Problem (MF-SG)$_\e$ is given by:
\bel{open-closed-re-e}
 u_\e = \Th_\e\big\{X_\e-\dbE[X_\e]\big\}+\bar\Th_\e\dbE[X_\e],
\ee
with $X_\e$ solving the closed-loop system:
$$\left\{\begin{aligned}
dX_\e(s) &=\big\{(A+B\Th_\e)(X_\e-\dbE[X_\e])+[(A+\bar A)+(B+\bar B)\bar\Th_\e]\dbE[X_\e]\big\}ds \\
  &\hp{=\ }+\big\{(C+D\Th_\e)(X_\e-\dbE[X_\e])+[(C+\bar C)+(D+\bar D)\bar\Th_\e]\dbE[X_\e]\big\}dW(s),\\
   X_\e(0) &=x.
\end{aligned}\right.$$

\ss
For any  $\e>0$ and $x\in\dbR^n$, the value of Problem (MF-SG)$_\e$ at $x$ is given by
$$
V_\e(x)=J_\e(x;u_1^\e,u_2^\e),
$$
where $u_\e=(u_1^\e,u_2^\e)$ is defined by \rf{open-closed-re-e}.
If the value  $V(x)$ of Problem (MF-SG) exists for some $x\in\dbR^n$,
then by the same arguments as  in the proof of \cite[Proposition 3.5]{Sun2020} we have
$$
\lim_{\e\to 0}V_\e(x)=V(x).
$$
This means that the family $\{u_\e\}_{\e>0}$ defined by \rf{open-closed-re-e} is an {\it approximate sequence} of Problem (MF-SG).
But as pointed out in \cite{Sun2020}, the existence of a value function does not imply that Problem (MF-SG) has an open-loop saddle point.
Thus in terms of the family $\{u_\e\}_{\e>0}$,  we present the following characterization for the open-loop solvability of Problem (MF-SG),
which is the main result of this section.

\begin{theorem}\label{thm:perturbation-approach}
Let {\rm\ref{ass:A1}--\ref{ass:A2}} and \rf{convex-condition1} hold.
Let $x\in\dbR^n$ be any given initial state and $\{u_\e\}_{\e>0}$ be the sequence defined by \rf{open-closed-re-e}.
Then the following statements are equivalent:
\begin{enumerate}[\rm(a)]
\item {\rm Problem} {\rm(MF-SG)} has an open-loop saddle point at $x$;
\item the family $\{u_\e\}_{\e>0}$ is bounded in the Hilbert space $L^2_\dbF(0,T;\dbR^{m_1+m_2})\equiv\cU_1\times\cU_2$, i.e.,
      $$\sup_{\e>0}\,\dbE\int_0^T|u_\e(s)|^2ds <\i;$$
\item the family $\{u_\e\}_{\e>0}$ is strongly convergent  in $L^2_\dbF(0,T;\dbR^{m_1+m_2})$ as $\e\to0$.
\end{enumerate}
Whenever {\rm(a)}, {\rm(b)}, or {\rm(c)} is satisfied, the strong  limit of  $\{u_\e\}_{\e>0}$
is an open-loop saddle point of {\rm Problem (MF-SG)} for the initial state $x$.
\end{theorem}

\begin{remark}
Since there is no coupled system in the above perturbation approach and
all the equations  involved can be solved by iteration method,
it will be much more convenient for computational purposes.
\end{remark}

The perturbation approach in stochastic LQ control problems was initially introduced by Sun--Li--Yong \cite{Sun-Li-Yong2016},
and further sharpened by Wang--Sun--Yong \cite{Wang-Sun-Yong} for finding the so-called weak closed-loop optimal strategies.
However, compared with the control problems \cite{Sun-Li-Yong2016,Wang-Sun-Yong},
in the game problem there are some new  difficulties, especially in proving the boundedness of $\{u_\e\}_{\e>0}$.
Before proving \autoref{thm:perturbation-approach}, we  present the following tailormade example,
from which we can perceive some essential differences between the perturbation approaches of LQ game and control problems.

\begin{example}\label{exap-perturbation1}
For any $x\in\dbR$, consider the  one-dimensional state equation
\bel{state-example1}\left\{\begin{aligned}
   \dot{X}(s) &=su_1(s)+u_2(s),\q s\in[0,1],\\
     X(0)&= x,
\end{aligned}\right.\ee
and the quadratic functional
\bel{cost-example1}
J(x;u_1,u_2)=-|X(1)|^2+\int_0^1s^2|u_1(s)|^2ds.
\ee
In the example, we let $\cU_1=\cU_2$ be the space of $\dbR$-valued square-integrable
functions on $[0,1]$. Note that
\begin{align*}
J(1;0,u_2)&=-|X(1)|^2=-\Big|1+\int_0^1u_2(s)ds\Big|^2\les 0=J(1;0,-1),\qq\forall u_2\in\cU_2;\\
J(1;u_1,-1)&=-\Big|\int_0^1 su_1(s)ds\Big|^2+\int_0^1 |su_1(s)|^2ds\ges 0=J(1;0,-1),\q\forall u_1\in\cU_1.
\end{align*}
Thus the control pair $(0,-1)$ is an open-loop saddle point for  $x=1$
and the convexity-concavity condition \rf{convex-condition1} holds.

\ms

For any $\e>0$ and $(u_1,u_2)\in\cU_1\times\cU_2$, denote
\bel{J-e-example1}
J_\e(x;u_1,u_2)=J(x;u_1,u_2)-\e \int_0^1|u_2(s)|^2ds.
\ee
Then
\begin{align}
&J_\e (0;0,u_2)= J(0;0,u_2) -\e \int_0^1|u_2(s)|^2ds\les-\e \int_0^1|u_2(s)|^2ds,\q\forall u_2\in\cU_2,\\
&J_\e(0;u_1,0)=J(0;u_1,0)\ges 0,\q\forall u_1\in\cU_1.
\end{align}
Roughly speaking, the above implies that the convexity of the mapping $u_1\mapsto J_\e(x;u_1,u_2)$ equals that of $u_1\mapsto J(x;u_1,u_2)$
and the concavity of  $u_2\mapsto J_\e(x;u_1,u_2)$ is stronger than that of $u_2\mapsto J(x;u_1,u_2)$.
However, we will show that the LQ game problem associated with \rf{state-example1}--\rf{J-e-example1}
has no open-loop saddle point for $x=1$.

\ms

We shall prove the above claim by contradiction.
Suppose that the LQ game problem associated with \rf{state-example1}--\rf{J-e-example1} has an
open-loop saddle point $(u_1^*,u_2^*)$. Then we must have
\bel{u2-star-1}
\int_0^1u_2^*(s)ds=-1.
\ee
Otherwise, $u^*_1$ is optimal for the LQ control problem associated with the state equation
\bel{state-example1-2}
   \dot{X}(s) =su_1(s),\q s\in[0,T];\q
     X(0)= h\deq 1+\int_0^1u_2^*(s)ds\neq 0,
\ee
and the cost functional $J(u_1)=J_\e(1;u_1,u_2^*)$. The corresponding optimal state is denoted by $X^*$.
Then by \cite[Corollary 3.3]{Sun-Li-Yong2016},  $u_1^*$ satisfies the following stationary condition:
\bel{sta-con-ex1}
sY^*(s)+s^2u_1^*(s)=0,\q s\in[0,1],
\ee
where $Y^*$ is the solution to the following adjoint equation:
\bel{adjoint-equation}
\dot{Y}^*(s)=0,\q s\in[0,T];\q Y^*(1)=-X^*(1).
\ee
By solving \rf{adjoint-equation}, the stationary condition \rf{sta-con-ex1} can be rewritten as
\bel{sta-con-ex11}
-sX^*(1)+s^2u_1^*(s)=0,\q s\in[0,1].
\ee
Recalling that $u_1^*\in\cU_1$ is a square-integrable function, we get
\bel{X-star}
X^*(1)=0,
\ee
which implies that
\bel{u-1-star}
u^*_1(s)=0,\q s\in[0,1].
\ee
Thus,
\bel{X-1-star}
0=X^*(1)=h+\int_0^1 su_1^*(s)ds=h.
\ee
This contradicts \rf{state-example1-2} and thus \rf{u2-star-1} holds.
Further, by \rf{u2-star-1},  note that
\begin{align}
J_\e (1;u_1,u^*_2)&=-\Big|\int_0^1 su_1(s)ds\Big|^2+\int_0^1 |su_1(s)|^2ds-\e \int_0^1|u^*_2(s)|^2ds,\nn\\
\label{J-u1-u2-star}
J_\e (1;u^*_1,u^*_2)&=\inf_{u_1\in\cU_1[0,1]}J_\e (1;u_1,u^*_2),
\end{align}
then we have
\bel{u-1-star2}
u_1^*(s)=0,\q s\in[0,1].
\ee
It follows that $u^*_2$ is   optimal for the control problem associated with the state equation
\bel{state-example1-31}
   \dot{X}(s) =u_2(s),\q s\in[0,1];\q
     X(0)=1 ,
\ee
and the cost functional $J(u_2)=-J_\e(1;0,u_2)$. Then, $u_2^*$ satisfies the stationary condition:
\bel{sta-con-ex1-21}
Y^*(s)+\e u_2^*(s)=0,\q s\in[0,1],
\ee
with $Y^*$ solving the adjoint equation:
\bel{adjoint-equation-22}
\dot{Y}^*(s)=0,\q s\in[0,1];\q Y^*(1)=X^*(1)=0.
\ee
It follows that
\bel{u-2-star-0}
u_2^*(s)=-{Y^*(s)\over\e}=0,\q s\in[0,1],
\ee
which contradicts \rf{u2-star-1}. Therefore,  the claim is proved; that is the LQ game problem associated
with \rf{state-example1}--\rf{J-e-example1} has no open-loop saddle points for $x=1$.
\end{example}

In the LQ control problems \cite{Sun-Li-Yong2016,Wang-Sun-Yong},
the perturbed cost functional $J_\e(x;u)$  is defined by adding $\e\|u\|^2$ to the
original one $J(x;u)$, where $u$ is the control process.
By the monotonicity of the mapping $\e \mapsto V_\e(x)$,
\cite{Sun-Li-Yong2016} showed that $\|u_\e\|^2$ is bounded by $ \|u^*\|^2$,
provided the original problem has an optimal control $u^*$.

\ms
%Let us explain the difficulty in establishing the perturbation approach of Problem (MF-SG)
%a little more carefully.
From \autoref{exap-perturbation1}, we see that the new game associated with
$J(x;u_1,u_2)-\e\|u_2\|^2$ possibly has no saddle point even if the original is open-loop solvable.
Thus in the perturbation approach of games,
adding both $\e\|u_1\|^2$ and $-\e\|u_2\|^2$ to the original functional turns out to be necessary,
due to which the value function $V_\e(x)$ is not monotone in $\e$.
Noticing this key point,
a seemingly feasible approach is to introduce two parameters $\e_1,\e_2>0$
and  consider the game problem with the functional:
$$
J_{\e_1,\e_2}(x;u_1,u_2)=J(x;u_1.u_2)+\e_1\|u_1\|^2-\e_2\|u_2\|^2.
$$
Since $J_{\e_1,\e_2}$ is monotone in each $\e_i$,
it seems that the convergence of  $\{u_{\e_1,\e_2}\}_{\e_1,\e_2>0}$
can be obtained by letting $\e_1\to 0$ and $\e_2\to 0$ separately.
%In other words, one may wish to study the  perturbation approach of game by considering two control problems.
%the limit has been taken by
But in fact, after letting $\e_1\to 0$, one cannot take the limit by letting $\e_2\to 0$,
because the game with the functional $J_{0,\e_2}$ is possibly unsolvable (see \autoref{exap-perturbation1}).
In conclusion, the perturbation approaches of  LQ games (i.e., \autoref{thm:perturbation-approach}) and controls
(i.e., \cite{Sun-Li-Yong2016,Wang-Sun-Yong})  are essentially different.

\subsection{Proof of \autoref{thm:perturbation-approach}}
With the preparations in Sections \ref{Sec:Pre}, \ref{Sec:Rep-Functional} and Subsection \ref{subsec:perturbation},
now we are ready to prove \autoref{thm:perturbation-approach} by a Hilbert space method,
in which Proposition \ref{Prop:operator} and the Mazur's theorem (see Yosida \cite[p.120, Theorem 2]{Yosida}) play  important roles.
%Especially, by the Hilbert space method, we transfer the

\ms\noindent
\emph{\textbf{Proof of \autoref{thm:perturbation-approach}.}}
(i) We begin by proving the implication  (a) $\Rightarrow$ (b).
Let $v^*$ be an open-loop saddle point of Problem (MF-SG) for the initial state $x$.
Then by \autoref{Thm:suffi-nece-condition}, $v^*$ must satisfy
\bel{v-star}
\cM v^*+\cK x=0,
\ee
where the operators $\cM$ and $\cK$ are defined by \rf{def-M}.
On the other hand,  \autoref{suffi-condition} shows that for any $\e>0$,
the unique open-loop saddle point $u_\e$ of Problem (MF-SG)$_\e$ for $x$ can be also given by
\bel{main-suffi-nece-condition-e}
u_\e=-\cM^{-1}_\e\cK x,
\ee
where
\begin{align}
&\cM_\e\deq\cM+\begin{pmatrix}\e I_{m_1}& 0\\0 & -\e I_{m_2}\end{pmatrix}
= \begin{pmatrix}\cM_{11}+\e I_{m_1}& \cM_{12}\\\cM_{21} & \cM_{22} -\e I_{m_2}\end{pmatrix}.
\end{align}
Combining \rf{v-star} with \rf{main-suffi-nece-condition-e} yields that
\bel{main-suffi-nece-condition-e1}
u_\e=-\cM^{-1}_\e\cK x=\cM^{-1}_\e\cM v^*.
\ee
Then by \autoref{Prop:operator}, noting \rf{convex-condition1}, we have
\bel{main-suffi-nece-condition-e2}
\|u_\e\|^2=\|\cM^{-1}_\e\cM v^*\|^2\les \|\cM^{-1}_\e\cM\|^2\|v^*\|^2\les \|v^*\|^2.
\ee
Thus, $\{u_\e\}_{\e>0}$ is bounded in the Hilbert space $L^2_{\dbF}(0,T;\dbR^{m_1+m_2})$.

\ms
(ii) We show that (b) $\Rightarrow$ (a). For convenience, we let
\bel{main-suffi-nece-condition-e4}
\sup_{\e>0}\|u_\e\|^2\les \|\bar u\|^2, \q\hbox{for some}~ \bar u\in\cU_1\times \cU_2.
\ee
Since $\{u_\e\}_{\e>0}$ is bounded in the Hilbert space $L^2_{\dbF}(0,T;\dbR^{m_1+m_2})$,
it admits a weakly convergent subsequence.
We denote this subsequence  by $\{u_{\e_j}\}_{j\ges 1}=\{(u_1^{\e_j},u_2^{\e_j})\}_{j\ges 1}$ and its weak limit by $u^*=(u_1^*,u_2^*)$.
Then by Mazur's theorem there exist $\l_{kj}\in[0,1]$, $j=1,2,...,N_k$  such that
\bel{Mazur1}
\sum_{j=1}^{N_k}\l_{kj}=1,\q k=1,2,...,
\ee
and
\bel{Mazur1-2}
\Big\|\sum_{j=1}^{N_k}\l_{kj}u_{\e_{k+j}}-u^*\Big\|\to 0,\q \hbox{as} \q k\to\i.
\ee
By the convexity of the mapping $u_1\mapsto J(x;u_1,u_2)$,
we have
\begin{align}
 &J\big(x;\sum_{j=1}^{N_k}\l_{kj}u_1^{\e_{k+j}},\sum_{j=1}^{N_k}\l_{kj}u_2^{\e_{k+j}}\big)
\les \sum_{j=1}^{N_k}\l_{kj}J\big(x;u_1^{\e_{k+j}},\sum_{j=1}^{N_k}\l_{kj}u_2^{\e_{k+j}}\big)\nn\\
\nn&\q =\sum_{j=1}^{N_k}\l_{kj}J_{\e_{k+j}}\big(x;u_1^{\e_{k+j}},\sum_{j=1}^{N_k}\l_{kj}u_2^{\e_{k+j}}\big)
-\sum_{j=1}^{N_k}\l_{kj}\e_{k+j}\|u_1^{\e_{k+j}}\|^2\\
&\qq+\sum_{j=1}^{N_k}\l_{kj}\e_{k+j}\Big\|\sum_{j=1}^{N_k}\l_{kj}u_2^{\e_{k+j}}\Big\|^2\nn\\
\nn &\q\les\sum_{j=1}^{N_k}\l_{kj}J_{\e_{k+j}}\big(x;u_1^{\e_{k+j}},\sum_{j=1}^{N_k}\l_{kj}u_2^{\e_{k+j}}\big)
+\sum_{j=1}^{N_k}\l_{kj}\e_{k+j}\Big\|\sum_{j=1}^{N_k}\l_{kj}u_2^{\e_{k+j}}\Big\|^2\nn.
\end{align}
By \rf{main-suffi-nece-condition-e4}, we get
$$
\Big\|\sum_{j=1}^{N_k}\l_{kj}u_2^{\e_{k+j}}\Big\|^2\les \sum_{j=1}^{N_k}\l_{kj}\|u_2^{\e_{k+j}}\|^2\les \|\bar u\|^2,
$$
and thus
\begin{align}
\nn &J\big(x;\sum_{j=1}^{N_k}\l_{kj}u_1^{\e_{k+j}},\sum_{j=1}^{N_k}\l_{kj}u_2^{\e_{k+j}}\big)\\
&\q\les\sum_{j=1}^{N_k}\l_{kj}J_{\e_{k+j}}\big(x;u_1^{\e_{k+j}},\sum_{j=1}^{N_k}\l_{kj}u_2^{\e_{k+j}}\big)
+\sum_{j=1}^{N_k}\l_{kj}\e_{k+j}\|\bar u\|^2 \label{con-u1}.
\end{align}
Moreover, note that $(u_1^{\e_{k+j}},u_2^{\e_{k+j}})$ is an open-loop saddle of Problem (MF-SG)$_{\e_{k+j}}$.
Thus for any $u_1\in\cU_1$,
\bel{con-u2}
J_{\e_{k+j}}\big(x;u_1^{\e_{k+j}},\sum_{j=1}^{N_k}\l_{kj}u_2^{\e_{k+j}}\big)
\les J_{\e_{k+j}}\big(x;u_1^{\e_{k+j}},u_2^{\e_{k+j}}\big)
\les J_{\e_{k+j}}\big(x;u_1,u_2^{\e_{k+j}}\big).
\ee
Substituting the above into \rf{con-u1} and then by the concavity of the mapping $u_2\mapsto J(x;u_1,u_2)$,
we have
\begin{align}
 &J\big(x;\sum_{j=1}^{N_k}\l_{kj}u_1^{\e_{k+j}},\sum_{j=1}^{N_k}\l_{kj}u_2^{\e_{k+j}}\big)
 \les \sum_{j=1}^{N_k}\l_{kj}J_{\e_{k+j}}\big(x;u_{1},u_{2}^{\e_{k+j}}\big)+\sum_{j=1}^{N_k}\l_{kj}\e_{k+j}\|\bar u\|^2\nn\\
&\q \les\sum_{j=1}^{N_k}\l_{kj}J\big(x;u_{1},u_2^{\e_{k+j}}\big)+\sum_{j=1}^{N_k}\l_{kj}\e_{k+j}\|u_{1}\|^2
+\sum_{j=1}^{N_k}\l_{kj}\e_{k+j}\|\bar u\|^2\nn\\
\label{con-u3}
&\q\les J\big(x;u_{1},\sum_{j=1}^{N_k}\l_{kj}u_2^{\e_{k+j}}\big)+\sum_{j=1}^{N_k}\l_{kj}\e_{k+j}\|u_{1}\|^2
+\sum_{j=1}^{N_k}\l_{kj}\e_{k+j}\|\bar u\|^2.
\end{align}
Thus by \rf{Mazur1}--\rf{Mazur1-2} and the (strong) continuity of the mapping $(u_1,u_2)\mapsto J(x;u_1,u_2)$,
letting $k\to\i$ in \rf{con-u3} yields that
\bel{con-u4}
J(x;u_1^*,u_2^*)\les J(x;u_1,u_2^*),\q\forall u_1\in\cU_1.
\ee
By the same argument as the above, we also have
\bel{con-u5}
J(x;u_1^*,u_2^*)\ges J(x;u^*_1,u_2),\q\forall u_2\in\cU_2.
\ee
The result then concludes from \rf{con-u4}--\rf{con-u5}.

\ms

(iii) The implication (c) $\Rightarrow$ (b) is trivially true.
We next prove  (b) $\Rightarrow$ (c).
We first claim: {\it The family $\{u_\e\}_{\e>0}$ is weakly convergent  as $\e\to0$
and the weak limit is an open-loop saddle point of {\rm Problem (MF-SG)} for $x$.}

\ms

If the above claim holds, then the family $\{u_\e\}_{\e>0}$ converges weakly to an open-loop saddle point $u^*$ of
Problem {\rm(MF-SG)} as $\e\to0$. Thus by the  weakly lower semicontinuity of the mapping $u\mapsto\|u\|^2$,
we have
\bel{t-3-51} \dbE\int_0^T|u^*(s)|^2ds \les \liminf_{\e\to0} \dbE\int_0^T|u_\e(s)|^2ds.\ee
On the other hand,  by \rf{main-suffi-nece-condition-e2},
with $v^*$ replaced by $u^*$, we obtain
\bel{t-3-5}\dbE\int_0^T|u_\e(s)|^2ds \les \dbE\int_0^T|u^*(s)|^2ds, \q\forall\e>0.\ee
Combining \rf{t-3-51}  with \rf{t-3-5} yields that
 $$
\lim_{\e\to 0} \dbE\int_0^T|u_\e(s)|^2ds=\dbE\int_0^T|u^*(s)|^2ds.
 $$
Recall that $u^*$ is the weak limit of $\{u_\e\}_{\e>0}$. Then the above implies that
$$
\lim_{\e\to0}\dbE\int_0^T|u_\e(s)-u^*(s)|^2ds=0.
$$
It shows that $\{u_\e\}_{\e>0}$ converges strongly to $u^*$ as $\e\to0$.
Thus to show (b) $\Rightarrow$ (c), we only need to prove the claim.

\ms
Noting that $L_{\dbF}^2(0,T;\dbR^{m_1+m_2})$ is a Hilbert space, to verify the claim,
it suffices to show that every weakly convergent subsequence of $\{u_\e\}_{\e>0}$
has the same weak limit, which is an open-loop saddle point of Problem {\rm(MF-SG)} for $x$.
Let $u^*$ and $\h u$ be the weak limits of two different weakly convergent subsequences
$\{u_{\e^i_k}\}_{k=1}^\infty$ $(i=1,2)$ of $\{u_\e\}_{\e>0}$.
Then by the same argument as in the proof of (b) $\Rightarrow$ (a),
we can show that both $u^*$ and $\h u$ are open-loop saddle points.
By the convexity of the mapping $u_1\mapsto J(x;u_1,u_2)$,
we have
\begin{align}
J\lt(x;{u_1^*+\h u_1\over 2},{u_2^*+\h u_2\over 2}\rt)
&\les {1\over2}J\lt(x;u_1^*,{u_2^*+\h u_2\over 2}\rt)+{1\over2}J\lt(x;\h u_1,{u_2^*+\h u_2\over 2}\rt).
\end{align}
Noting that both $u^*$ and $\h u$ are open-loop saddle points,
and  the mapping $u_2\mapsto J(x;u_1,u_2)$ is concave, we have
\begin{align}
\nn&{1\over2}J\lt(x;u_1^*,{u_2^*+\h u_2\over 2}\rt)+{1\over2}J\lt(x;\h u_1,{u_2^*+\h u_2\over 2}\rt)\\
\nn&\q\les {1\over2}J\lt(x;u_1^*,u_2^*\rt)
+{1\over2}J\lt(x;\h u_1,\h u_2\rt)\\
\nn&\q\les {1\over2}J\lt(x;u_1,u_2^*\rt)
+{1\over2}J\lt(x;u_1,\h u_2\rt)\\
&\q\les J\lt(x;u_1,{u_2^*+\h u_2\over 2}\rt),\q\forall u_1\in\cU_1.
\end{align}
Thus,
\bel{u-star-hat1}
J\lt(x;{u_1^*+\h u_1\over 2},{u_2^*+\h u_2\over 2}\rt)
\les J\lt(x;u_1,{u_2^*+\h u_2\over 2}\rt),\q\forall u_1\in\cU_1.
\ee
Similarly, we can  prove
\bel{u-star-hat2}
J\lt(x;{u_1^*+\h u_1\over 2},{u_2^*+\h u_2\over 2}\rt)
\ges J\lt(x;{u_1^*+\h u_1\over 2},u_2\rt),\q\forall u_2\in\cU_2.
\ee
Combining \rf{u-star-hat1} with \rf{u-star-hat2}, we get that ${u^*+\h u\over 2}$
is also an open-loop saddle point of Problem (MF-SG) with respect to $x$.
Thus by \rf{main-suffi-nece-condition-e2},
with $v^*$ replaced  by ${u^*+\h u\over 2}$, we obtain
\bel{proof11} \dbE\int_0^T|u_{\e^i_k}(s)|^2ds \les \dbE\int_0^T\lt|{u^*(s)+\h u(s)\over 2}\rt|^2ds, \q i=1,2. \ee
By the weakly lower semicontinuity of the mapping $u\mapsto\|u\|^2$ again, we have
\begin{align*}
\dbE\int_0^T|u^*(s)|^2ds \les \liminf_{k\to\i}  \dbE\int_0^T|u_{\e^1_k}(s)|^2ds,\\
 \dbE\int_0^T|\h u(s)|^2ds \les \liminf_{k\to\i}  \dbE\int_0^T|u_{\e^2_k}(s)|^2ds.
 \end{align*}
Thus  by taking inferior limits on the both sides of \rf{proof11}, we get
\begin{align*}
\dbE\int_0^T|u^*(s)|^2ds \les \dbE\int_0^T\lt|{u^*(s)+\h u(s)\over 2}\rt|^2ds,\\
\dbE\int_0^T|\h u(s)|^2ds \les \dbE\int_0^T\lt|{u^*(s)+\h u(s)\over 2}\rt|^2ds,
\end{align*}
which implies that
%
%$$ 2\lt[\dbE\int_0^T|u^*(s)|^2ds+\dbE\int_0^T|\h u(s)|^2ds\rt] \les \dbE\int_0^T|u^*(s)+\h u(s)|^2ds, $$
%
%which is equivalent to
%
$$ \dbE\int_0^T|u^*(s)-\h u(s)|^2ds\les 0. $$
The claim is established.$\hfill\qed$

\begin{remark}
The boundedness of $\|\cM_\e^{-1}\cM\|^2$ is sufficient for that of $\|u_\e\|^2$,
which implies that $\{u_\e\}_{\e>0}$ admits a weakly convergent  subsequence.
With the help of  Mazur's theorem, the (strong) convergence of $\{u_\e\}_{\e>0}$
follows from the explicit upper bound (which is exactly $1$) of $\|\cM_\e^{-1}\cM\|^2$.
\end{remark}

\begin{remark}
If Problem (MF-SG) reduces to a (mean-field) LQ optimal control problem,
then $\cM=\cM_1$ and \rf{main-suffi-nece-condition-e1} becomes
$$
u_\e=\cM_{1\e}^{-1}\cM_1 v^*=(\cM_1+\e I)^{-1}\cM_1 v^*.
$$
Note that $\cM_1$ is a positive operator, while $\cM$ is indefinite in general.
Then the  Hilbert space method  brings the following new viewpoint: The perturbation approaches of LQ controls and games
are the outcomes of the explicit norm estimates for perturbed positive  operators and indefinite operators, respectively.
\end{remark}

%\section{Applications}
%\subsection{Zero-sum game with mean-variance functional}
\section{Example}\label{Sec:example}
In this section, we present a simple example to illustrate the procedure for finding the open-loop saddle points
 by \autoref{open-loop-sovability} under the sufficient condition \rf{uniform-convex-condition};
and identifying the open-loop solvability by \autoref{thm:perturbation-approach} under the necessary condition \rf{convex-condition}.

\begin{example}\label{example22}
Consider the  one-dimensional state equation
\bel{2state-example1}\left\{\begin{aligned}
   dX(s) &=u_1(s)ds+u_2(s)dW(s),\q s\in[0,1],\\
     X(0)&= x,
\end{aligned}\right.\ee
and the quadratic functional
\bel{2cost-example1}
J(x;u_1,u_2)=\dbE\Big\{-|X(1)|^2+\int_0^1\(|u_1(s)|^2-|\dbE[u_2(s)]|^2\)ds\Big\}.
\ee
It is straightforward to see that
\bel{2-convex-concave}
J(0;u_1,0)\ges 0\q\hbox{and}\q  J(0;0,u_2)\les 0,\q \forall (u_1,u_2)\in\cU_1\times\cU_2.
\ee
Suppose that $(u_1^*,u_2^*)$ is an open-loop saddle point, then by \autoref{lemma-optimality},
$(u_1^*,u_2^*)$ must satisfy
$$
u_1^*=-Y,\q \dbE[u_2^*]=-Z,
$$
with
\bel{OS-E}\left\{\begin{aligned}
&dX(s)=-Y(s)ds-Z(s)dW(s),\\
&dY(s)=Z(s)dW(s),\\
&X(0)=x,\q Y(T)=-X(T).
\end{aligned}\right.
\ee
Note that the solution of the Riccati equation associated with \rf{OS-E} is $-{1\over s}$,
which is not integrable over $[0,1]$, thus the decoupling technique is  not applicable.
Moreover,  FBSDE \rf{OS-E} does not satisfy the so-called {\it monotone condition} in Hu--Peng \cite{Hu-Peng1995}.
Thus we also cannot determine the solvability  of \rf{OS-E} by \cite{Hu-Peng1995} directly.

\ms
In the following, let us apply \autoref{thm:perturbation-approach} to determine the open-loop solvability of the game.
For any $\e>0$, we denote
\bel{2cost-e}
J_\e(x;u_1,u_2)=J(x;u_1,u_2)+\e\dbE\int_0^1|u_1(s)|^2ds-\e\dbE\int_0^1|u_2(s)|^2ds.
\ee
By \rf{2-convex-concave}, we have
\bel{2-uniform-convex-concave}
\begin{aligned}
&J_\e(0;u_1,0)\ges \e\dbE\int_0^1|u_1(s)|^2ds,\q~~ \forall u_1\in\cU_1,\\
&J_\e(0;0,u_2)\les -\e\dbE\int_0^1|u_2(s)|^2ds,\q \forall u_2\in\cU_2.
\end{aligned}\ee
Then we can apply \autoref{open-loop-sovability} to find the unique open-loop saddle point
of Problem (MF-SG)$_\e$ with state equation \rf{2state-example1} and functional \rf{2cost-e}.
The corresponding Riccati equations \rf{Ric1-e}--\rf{Ric2-e} in the example read:
\bel{ex-Pe}\left\{\begin{aligned}
&\dot{P_\e}(s)-{P_\e(s)^2\over 1+\e}=0,\q s\in[0,1],\\
&\dot{\Pi_\e}(s)-{\Pi_\e(s)^2\over 1+\e}=0,\q s\in[0,1],\\
&P_\e(1)=-1,\q \Pi_\e(1)=-1.
\end{aligned}\right.
\ee
Solving \rf{ex-Pe} by separating variables, we get
\bel{Pe-Pie}
P_\e(s)=-{1+\e \over s+\e},\q \Pi_\e(s)=-{1+\e \over s+\e},\q s\in[0,1].
\ee
Define the corresponding feedback operators $\Th_\e$ \rf{def-TH-e} and $\bar\Th_\e$ \rf{def-bar-Th-e}  by
\bel{Th-e-bar=Th-e}
\Th_\e(s)=\begin{pmatrix} {1\over s+\e} \\ 0\end{pmatrix},\q
\bar\Th_\e(s)=\begin{pmatrix} {1 \over s+\e} \\ 0\end{pmatrix},\q s\in[0,1].
\ee
Then the unique open-loop saddle point of Problem (MF-SG)$_\e$ is given by
\bel{2-u-star-e}
u_\e=\begin{pmatrix} u_1^\e \\ u^\e_2\end{pmatrix}=\Th_\e\big\{X_\e-\dbE[X_\e]\big\}+\bar\Th_\e\dbE[X_\e]
= \begin{pmatrix} {X_\e\over s+\e} \\0\end{pmatrix},\q s\in[0,1],
\ee
with $X_\e$ being the unique solution to the following closed-loop system:
\bel{2state-e-example1}\left\{\begin{aligned}
   dX_\e(s) &={X_\e(s)\over s+\e}ds,\q s\in[0,1],\\
     X_\e(0)&= x.
\end{aligned}\right.\ee
By the variation of constants formula for ordinary differential equations, we get
\bel{X-e-solved}
X_\e(s)={s+\e\over \e}x,\q s\in[0,1].
\ee
Combining the above with \rf{2-u-star-e},
we obtain the following explicit representation of $u_\e$:
\bel{2-u-star-e1}
u_\e=\begin{pmatrix} u^\e_1 \\ u^\e_2\end{pmatrix}
= \begin{pmatrix} {x\over \e} \\0\end{pmatrix},\q s\in[0,1].
\ee
Moreover, note that
\bel{u-e-bound}
\sup_{\e>0}\dbE\int_0^1|u_\e(s)|^2=\i, \q \hbox{if}\q x\neq 0,
\ee
and
\bel{u-e-bound1}
\sup_{\e>0}\dbE\int_0^1|u_\e(s)|^2=0, \q \hbox{if}\q x= 0.
\ee
Thus according to \autoref{thm:perturbation-approach}, the problem
has no open-loop saddle point for $x\neq 0$ and has an open-loop saddle point $(0,0)$ for $x=0$.
\end{example}

\section{Appendix}

\emph{\textbf{Proof of \autoref{lemma-algebra}}}.
The proof is divided into four cases.

\ms
\textbf{Case 1. } If $K$ is invertible and $M$ is positive definite, then
\begin{align}
\nn &L^\top ML-L^\top M K(K^\top M K+\d I_m)^{-1}K^\top ML\\
&\label{lemma-algebra-1}\q\ges L^\top ML^\top-L^\top M K(K^\top M K)^{-1}K^\top ML=0.
\end{align}

\ss
\textbf{Case 2. } If $m=n$, then there exist a sequence of invertible matrices $\{K_\e\}_{\e>0}$
and a sequence of positive definite matrices $\{M_\e\}_{\e>0}$ such that
\bel{lemma-algebra-2}
\lim_{\e\to 0} K_\e=K \q\hbox{and}\q\lim_{\e\to 0} M_\e=M.
\ee
Then from the facts $M\ges 0$, $\d>0$ and the result obtained in Case 1, we have
\begin{align}
\nn &L^\top ML-L^\top M K(K^\top M K+\d I_m)^{-1}K^\top ML\\
&\label{lemma-algebra-3}\q=\lim_{\e\to 0}\big[ L^\top M_\e L-L^\top M_\e K_\e(K_\e^\top M_\e K_\e+\d I_m)^{-1}K_\e^\top M_\e L\big]\ges 0.
\end{align}

\ss
\textbf{Case 3. } If $n>m$, set $\h K=(K, \textbf{0})$ such that $\h K\in\dbR^{n\times n}$,
where $\textbf{0}$ is the zero matrix with an appropriate dimension.
Then by the results obtained in Case 2, we have
\begin{align}
\nn 0 &\les L^\top ML-L^\top M \h K(\h K^\top M \h K+\d I_n)^{-1}\h K^\top ML\\
&\nn=L^\top ML-\begin{pmatrix} L^\top M  K & \textbf{0}\end{pmatrix}\begin{pmatrix} K^\top M K+\d I_m & \textbf{0}\\
 \textbf{0} &\d I_{n-m}\end{pmatrix}^{-1}\begin{pmatrix} K^\top M L\\
 \textbf{0}\end{pmatrix}\\
 &\label{lemma-algebra-4}=L^\top ML-L^\top M K(K^\top M K+\d I_m)^{-1}K^\top ML.
\end{align}

\ss
\textbf{Case 4. } If $n<m$, set
$$\bar L=\begin{pmatrix} L & \textbf{0}\\
 \textbf{0} &\textbf{0}\end{pmatrix},~
 \bar M=\begin{pmatrix} M & \textbf{0}\\
 \textbf{0} & \textbf{0}\end{pmatrix},~
 \bar K=\begin{pmatrix} K \\
 \textbf{0} \end{pmatrix},
 $$
such that $\bar L,\bar K\in\dbR^{m\times m}$ and $\bar M\in\dbS_+^{ m}$.
By the results obtained in Case 2 again, we have
\begin{align}
\nn 0 &\les\bar L^\top\bar  M\bar L-\bar L^\top \bar M \bar K(\bar K^\top\bar M \bar K+\d I_m)^{-1}\bar K^\top\bar M\bar L\\
&\nn=\begin{pmatrix} L^\top M L& \textbf{0}\\
 \textbf{0} &\textbf{0}\end{pmatrix}
 -\begin{pmatrix} L^\top M K\\\textbf{0}\end{pmatrix}(K^\top M K+\d I_m)^{-1}\begin{pmatrix} K^\top M L& \textbf{0}\end{pmatrix}\\
 &\label{lemma-algebra-5}
 =\begin{pmatrix} L^\top ML-L^\top M K(K^\top M K+\d I_m)^{-1}K^\top ML& \textbf{0}\\
 \textbf{0} &\textbf{0}\end{pmatrix},
\end{align}
which implies that
$$
L^\top ML-L^\top M K(K^\top M K+\d I_m)^{-1}K^\top ML\ges 0.
$$

\section*{Acknowledgements} The authors would like to thank the associate editor and the anonymous referees for
their suggestive comments, which lead to this improved version of the paper.

\end{document}